\documentclass[11pt]{article}
\usepackage{graphicx,wrapfig,lipsum}
\usepackage{amsmath,amssymb,graphicx,authblk} 
\usepackage{amsfonts,amsthm}
\usepackage[margin=0.9in]{geometry} 
\usepackage{etoolbox}

\usepackage{enumerate}
\usepackage[toc]{appendix}
\usepackage{hyperref}
\usepackage{fullpage}
\usepackage{tikz}
\usepackage[many]{tcolorbox}
\usepackage{bbm}
\usepackage{comment}
\newtheorem{definition}{Definition}[section]
\newtheorem{theorem}{Theorem}
\newtheorem{proposition}{Proposition}
\newtheorem{lemma}{Lemma}

\usepackage{comment}
\usepackage{enumitem}
\usepackage[utf8]{inputenc}
\usepackage{xcolor}
\usepackage{graphicx}
\usepackage{subcaption}
\usepackage{algorithm}
\usepackage{algpseudocode}
\usepackage{tikz}
\usepackage{nicefrac}
\usepackage[capitalise,noabbrev]{cleveref}
\usepackage{tikz}
\usepackage[normalem]{ulem}

\makeatletter
\patchcmd{\mmeasure@}{\measuring@true}{
  \measuring@true
  \ifnum-\linenopenaltypar>\interdisplaylinepenalty
    \advance\interdisplaylinepenalty-\linenopenalty
  \fi
  }{}{}
\makeatother

\newcommand{\timeint}{(0,T)}

\newcommand{\di}[1]{\,\mathrm{d}#1}

\newtcolorbox{mybox}[1]{%
    tikznode boxed title,
    enhanced,
    arc=0mm,
    interior style={white},
    attach boxed title to top left= {yshift=-\tcboxedtitleheight/2-0.05cm, xshift=0.7cm},
    fonttitle=\small\bfseries,
    colbacktitle=white,coltitle=black,
    boxed title style={size=small,colframe=white,boxrule=0pt},
    title={#1}}

	\title{Strongly Coupled Two-scale System with Nonlinear Dispersion:\\ Weak Solvability and Numerical Simulation}

\author[a]{Vishnu Raveendran\footnote{raveendr@ins.uni-bonn.de}}
\author[b]{ Surendra Nepal}
\author[c]{Rainey Lyons}
\author[b]{ Michael Eden}
\author[b]{Adrian Muntean}
\affil[a]{Institute for Numerical Simulation, University of Bonn, Germany}
\affil[b]{Department of Mathematics and Computer Science, Karlstad University, Sweden}
\affil[c]{Department of Applied Mathematics, University of Colorado Boulder, USA}

	\date{\today} 

\graphicspath{ {./images/} }
\allowdisplaybreaks

\allowdisplaybreaks
\def\es{\varepsilon}
\begin{document}


\maketitle

\begin{abstract}\label{abstract}
	We investigate a two-scale system featuring an upscaled parabolic dispersion-reaction equation intimately linked to a family of elliptic cell problems.
    The system is strongly coupled through a dispersion tensor, which depends on the solutions to the cell problems, and via the cell problems themselves, where the solution of the parabolic problem interacts nonlinearly with the drift term.    
    This particular mathematical structure is motivated by a rigorously derived upscaled reaction-diffusion-convection model that describes the evolution of a population of interacting particles pushed by a large drift through an array of periodically placed obstacles (i.e., through a regular porous medium). 

    We prove the existence and uniqueness of weak solutions to our system by means of an iterative scheme, where particular care is needed to ensure the uniform positivity of the dispersion tensor.  
    Additionally, we use finite element-based approximations for the same iteration scheme to perform multiple simulation studies. Finally, we highlight how the choice of micro-geometry (building the regular porous medium) and of the nonlinear drift coupling affects the macroscopic dispersion of particles.  
\end{abstract}
\section{Introduction}\label{section_introduction}
The transport of substances through porous media typically involves a combination of drift, diffusion, and adsorption processes, all of which take place at the scale of the heterogenous porous structure. 
The interplay between these physical processes and the geometry of the underlying heterogeneous structure is intricate. 
Only in rare circumstances can one obtain useful insight into the quantitative understanding of the effective dispersion mechanisms (responsible for an eventual macroscopic migration of substances), the overall storage capacity of a given heterogeneous medium, or the resulting turbulent diffusion and anomalous forms of dissipation; see e.g. \cite{mclaughlin1985convection} for a discussion of simple turbulent flows conveying microstructure information.
However, understanding these phenomena is crucial for a number of modern technological applications, including the design of drug delivery systems, remediation of groundwater contamination, design of efficient hydrogen storage systems, and creation of polymer-based morphologies to facilitate current transport for organic solar cells; we refer the reader to \cite{Donovan_PLOS,hydrogen,PNAS} for examples of such applications. 

In this work, we are primarily interested in understanding what effects are observable at a macroscopic level, i.e. in terms of the effective indicators, within the interplay between fast drift and diffusion.
In other words, we are revisiting the Taylor-Aris notion of dispersion now taking place within a specific type of heterogeneous medium\footnote{G. I. Taylor tells in his 1954 paper \cite{taylor1954conditions} that shear flow smears out the concentration distribution, enhancing the rate at which it spreads in the direction of the flow. R. Aris confirms that the effect takes place at large Peclet numbers(see \cite{aris1956dispersion}).}.
As a first step, we are neither looking at processes related to sorption/adsorption mechanisms nor to surface chemical reactions, even though we are aware that such interface processes take place in real materials and can have an effect on the macroscopic response.  
We are, of course, not the first ones to look at Taylor-Aris dispersion for spatially structured domains.
The majority of the previous approaches aimed to justify that the macroscopic transport mechanisms can be effectively modeled via a constant dispersion tensor which encodes both the geometric properties of the underlying microstructure (like porosity and connectivity) and the diffusivity of the solute. 
Such effective models for transport in porous media are usually derived via some type of scale analysis; this can be the reference elementary volume (REV) approach often utilized in the engineering community \cite{bear1988dynamics, bear2012phenomenological,Wood2003RVE}, or the more rigorous route of mathematical homogenization \cite{ALLAIRE20102292,Allaire_dispersion,Amaziane}. 
In these works, the resulting limit problem is structurally identical to a diffusion problem and the constant symmetric effective dispersion can be directly calculated via cell problems.
However, in some situations, this perspective is too simplistic as the physical reality is more complicated.
For instance, earlier investigations have shown that the resulting dispersion tensor may not be symmetric \cite{Auriault2010} and that it can depend  (linearly or nonlinearly) on the solute concentration \cite{Schotting1999}. 
Such effects cannot be captured in the aforementioned averaging approaches. 

To this end, we are considering a dispersion-reaction equation governing the effective solute transport and production in a porous medium where the dispersion depends nonlinearly on the solute concentration,
\begin{subequations}\label{eq:introduction_problem}
\begin{equation}\label{eq:introduction_nonlinearpde}
\partial_tu+\text{div}(-D^*(W)\nabla u)=f\quad \text{in}\ \timeint\times\Omega,
\end{equation}
over some macroscopic domain $\Omega\subset\mathbb{R}^2$ representing a porous medium.
While \cref{eq:introduction_nonlinearpde} is macroscopic, the dispersion tensor $D^*(W)$ depends on microscopic solutions $W=(w_1,w_2)$ of cell problems over the fluid part of the porous microstructure $Y$ via
\begin{equation}\label{eq:introduction_nonlinearpde_cell}
 \text{div}_y\left(-D \nabla_y w_i+G_i(u)B w_i\right)=\text{div}_y\left(D e_i\right)\quad \text{in} \ Y,\quad (i=1,2).
 \end{equation}
\end{subequations}
Here, $D$ denotes the solute diffusivity, $B$ is the velocity field representing the movement of the solvent, and $G_i(u)$ represents the drift coupling.
Via the product $G_i(u)Bw_i$, the macroscopic function $u$ nonlinearly interacts with the cell problem solution $W$. 
Without this specific interaction in the cell problems, i.e., taking the case $G_i(u)\equiv const.$, the dispersion tensor assumes the standard form of linear dispersion, see \cite{Allaire_dispersion}.
Moreover, the case $G_i(u)\equiv0$ corresponds to pure diffusion \cite{Allaire92}.
\newpage

{The structure of the micro-macro model is motivated as the homogenization limit of a nonlinear drift problem of the type
\begin{equation}\label{Eq:orgin}
       \partial_t u^{\varepsilon} +\mathrm{div}(-D^{\varepsilon}\nabla u^{\varepsilon}+ \frac{1}{\es}B^{\varepsilon}P(u^{\varepsilon}))=f^{\varepsilon}
\end{equation}
defined in a periodic porous media $\Omega_\varepsilon\subset\mathbb{R}^2$ saturated with an incompressible fluid where $0<\varepsilon\ll1$ is a small scale parameter representative of this periodic porous geometry. 
The unknown function $u^\es$ represents the concentration profile of a population of interacting particles dissolved in the fluid.
These particles are transported through the porous medium by $(i)$ diffusion governed by the diffusivity $D^\es$, which is positive definite and $\es$-periodic, and $(ii)$ nonlinear advection of the form $\frac{1}{\es}B^{\varepsilon}P(u^{\varepsilon})$.
Here, $B_\varepsilon\colon\Omega_\varepsilon\to\mathbb{R}^2$ denotes the velocity profile of the fluid and is assumed to be known and $\es$-periodic.
The particular scaling with $\es^{-1}$ is the reason why such problems are called fast drift problems (cf.\cite{Allaire_dispersion,allaire2016,ijioma2019fast}), while the function $P(r)$ models the potentially nonlinear interactions of the particles with the moving fluid.
In particular, the case $P(r)=r$ corresponds to standard linear advection and $P(r):=r(1-r)$ is the result of a hydrodynamic limit of a \textit{totally asymmetric simple exclusion process} (TASEP) for interacting particles traversing a porous medium; for details concerning the derivation of this specific nonlinearity, we refer the reader to \cite{CIRILLO2016436,raveendran2024scaling,cirillo2020upscaling,raveendran21}.
At the internal boundary of the porous media, we assume the following Robin-type boundary condition:
\begin{equation}
    (-D^{\varepsilon}\nabla u^{\varepsilon}+ \frac{1}{\es}B^{\varepsilon}P(u^{\varepsilon}))\cdot n_\es=0.
\end{equation}
}
\noindent
{The upscaling of the aforementioned problem can be carried out by using the technique of two-scale convergence with drift, we refer to \cite{raveendran2023homogenization} for more details regarding this limiting process. 
In the above-mentioned scenario where the nonlinearity takes the form $P(r):=r(1-r)$, we have $G_i(r)=1-2r$ in \cref{eq:introduction_nonlinearpde_cell}
It is worth noting that here $G_i(r)=P'(r)$ emerges as a direct consequence of the homogenization process as outlined in \cite{raveendran2023homogenization}.
Since nonlinear interactions in the diffusion/dispersion part appear quite often in upscaled models, e.g., \cite{marusik2005,allaire2016,ijioma2019fast,bringedal2020phase}, we study the case where the $G_i(u)$ are general nonlinear functions while maintaining a simplified structure for the nonlinear dispersion term.}

It is worth noting that our scientific questions are very much  in the spirit of \cite{Goudon}, where the authors ask: what is the macroscopic response of  
\begin{equation}\label{Eq:Goudon}
       \partial_t u^{\varepsilon} +\mathrm{div}(-\eta\nabla u^{\varepsilon} +
       u^\varepsilon V^\varepsilon)=f^{\varepsilon},
\end{equation} where $\eta>0$ and $V^\varepsilon$ is some large random field? 
In other words, what precisely is the term $\langle uV\rangle$ in the corresponding upscaled equation
\begin{equation}\label{Eq:Goudon_averaged}
       \partial_t \langle u\rangle +\mathrm{div}(-\eta\nabla \langle u \rangle +
       \langle uV\rangle)=\langle f \rangle,
\end{equation}
where $\langle\cdot\rangle$ denotes some suitable average, and how can this be computed numerically. 
Comparing \eqref{Eq:Goudon} and \eqref{Eq:Goudon_averaged} with \eqref{Eq:orgin} and \eqref{eq:introduction_nonlinearpde}, we see that the vector field $V^\varepsilon$ is in our setting what the TASEP requires, and hence, the meaning of $\langle uV\rangle$ is for us precisely $-D^*(W)\nabla u$. 

System \eqref{eq:introduction_problem} is an example of a two-scale system, where we refer to the dispersion equation \eqref{eq:introduction_nonlinearpde} as the macroscopic equation and the cell problems \eqref{eq:introduction_nonlinearpde_cell} as the microscopic equations.
Such systems are also referred to as distributed-microstructure models, terminology introduced by R. E. Showalter in  \cite{showalter1993distributed}.
The presence of nonlinear coupling in the dispersion tensor has been previously studied.
For instance, in \cite{EDEN2022103408}, the authors looked at the nonlinear effect of colloids deposition on diffusion in porous media and established local-in-time existence (see also \cite{Adrian_siap,nikolopoulos2023multiscale}). 
In the same context, in \cite{Wiedemann2023} a similar model was analyzed via homogenization techniques.
Additionally, a related homogenized model for reactive flows in porous media was derived in \cite{allaire2016} {in which the authors derived the upscaled model as a two-scale system with nonlinear dispersion}. {We remark that the upscaled model for the large convection problem in \cite{Amaziane} 
 has some similarities to our model but differs structurally from the dispersion tensor discussed in our work, primarily due to two reasons. First, the scaling of the microscopic problem in \cite{Amaziane} introduces an $\varepsilon$ scaling in front of the diffusion coefficient, unlike the scaling used in our Equation \eqref{Eq:orgin}. Second, the nonlinear term 
$P(r)$ in Equation \eqref{Eq:orgin} plays a critical role in shaping the cell problem during homogenization, leading to a different macroscopic dispersion tensor. }

The simulation of such scale separated problems often comes with unique computational challenges. 
For instance, the nonlinear drift coupling in the cell problems requires solving \cref{eq:introduction_nonlinearpde_cell} many times. 
Well-established numerical techniques capable of handling such multiscale features are reported in \cite{allaire2005multiscale,henning2014adaptive, abdulle2012heterogeneous,bris2019multiscale}.
 Some of the closer works that explore the simulation of two-scale systems similar to ours are \cite{lind2020semidiscrete,olivares2021two,nikolopoulos2023multiscale}.
 In \cite{lind2020semidiscrete}, the authors address a two-scale coupled  problem where the coupling is via the reaction rate of macroscopic equation.
 On the other hand, authors in \cite{olivares2021two} study and simulate a two-scale system for a phase-field model for precipitation and dissolution in porous media via an iterative scheme and, in  \cite{nikolopoulos2023multiscale}, clogging of a porous medium via aggregation of colloidal particles was explored.
\par This paper is structured as follows:
In \cref{section_setting}, we present the mathematical model of interest, the geometric setup, as well as the assumptions on coefficients and data needed for the mathematical analysis.
This is followed by \cref{section_analysis} where the main result, namely the existence of solutions (\cref{theorem_existance_of_p(omega)}), is proven via an iteration scheme.
In \cref{section_simulation}, we use this iteration scheme to numerically simulate different scenarios, showcasing that the proposed two-scale model is able to capture interesting effects at both the microscopic and macroscopic scale. 
The results reported in this section are only preliminary, as the numerical analysis and further numerical exploration of our system are postponed for follow-up work.  
Finally in \cref{section_conclusion}, we conclude our work with a short summary of the findings and with a discussion of possible future investigations.

\section{Setting of the problem}\label{section_setting}
Let $T>0$ represent the time horizon.
We denote by $\Omega \subset \mathbb{R}^2$ a non-empty open bounded domain with $C^{2+\alpha}$ boundary, $\partial \Omega $, for some $0<\alpha<1$, and by $Y_0 \subset (0,1)^2 $ a compact set with positive measure and Lipschitz boundary such that the microscopic domain $Y:=(0,1)^2\backslash Y_0$ is connected.
We let $t\in \timeint$ represent the time variable, $x\in \Omega$ the macroscopic space variable, $y\in Y$ the microscopic space variable, $\Gamma_N=\partial Y_0$, and $n_y$ the outward unit normal vector across the interface $\partial Y$. 
With the sets $\Omega$ and $Y$ at hand, we can build a two-scale geometry. 
For a schematic representation of such a geometry, see \cref{Fig:GeometrySchem}.
We take $e_i$, $i=1,2$, to denote the standard unit basis vectors in $\mathbb{R}^2$.

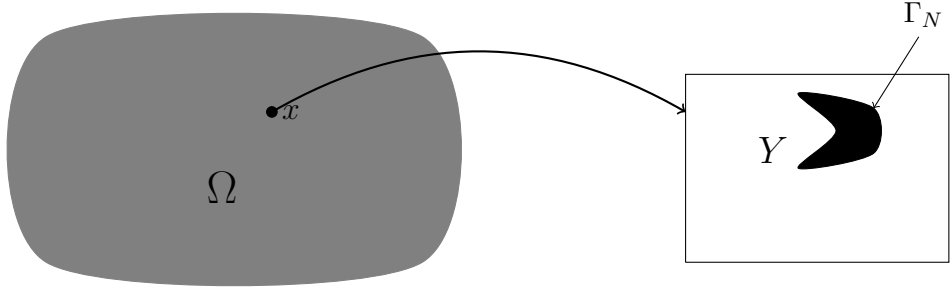
\begin{figure}[h]
    \begin{center}
        \begin{tikzpicture}
        \filldraw [gray] plot [smooth cycle] coordinates {(0,0) (5,0) (5,3) (0,3) };
\path[thick,->] (3.0,2)  edge [bend left] (8.5,2);
\filldraw[black] (3,2) circle (2pt) node[anchor=west]{\large $x$};
\filldraw[black] (2,1) circle (0pt) node[anchor=west]{\LARGE $\Omega$};
\draw (8.5,0) node [anchor=north] {{\scriptsize }} to (12,0) node [anchor=north] {{\scriptsize }}  to (12,2.5)node [anchor=south] {{\scriptsize }} to (8.5,2.5)node [anchor=south] {{\scriptsize }} to (8.5,0);
			\filldraw [black] plot [smooth cycle] coordinates {(10.5,1.75) (10,1.25) (11,1.45) (11,2.05)(10,2.25)};
	\draw[<-](11,2.05) to (11.6,3) node[anchor=south] {{ $\Gamma_{N}$}};
   \filldraw[black] (10,1.5) circle (0pt) node[anchor=east]{\Large $Y$};
\end{tikzpicture}
\caption{\label{Fig:GeometrySchem}Typical two-scale geometry: schematic representation of the macroscopic domain $\Omega$ and of the microscopic domain $Y$ with internal boundary $\Gamma_N$.}
    \end{center}
\end{figure}

We study the following nonlinear two-scale system: 
\begin{subequations}\label{Eq:Main_System}
\begin{mybox}{Problem $P(\Omega)$}
Find the unknown functions $u$ and $W:= (w_1,w_2)$ satisfying
 \begin{align}
    \partial_t u +\mathrm{div}(-D^*(W)\nabla  u)&=f &\mbox{in}& \hspace{.2cm}  \timeint\times \Omega, \label{homeq1}\\
    u&=0 &\mbox{on}& \hspace{.3cm}  \timeint\times \partial\Omega,\\
    u(0)&=g &\mbox{in}& \hspace{.3cm}  \overline{\Omega},\\
    \vspace{1cm}\nonumber\\
    \mathrm{div}_y\left(-D \nabla_y w_i+G_i(u)B w_i\right)&=\mathrm{div}_y (D e_i) &\mbox{in}& \hspace{.2cm}  Y,\label{cell3}\\
    \left( -D\nabla_y w_i+BG_i(u)w_i\right)\cdot n_y &= \left( De_i\right)\cdot n_y &\mbox{on}& \hspace{.2cm} \Gamma_N,\label{cellbn3}\\
    w_i\, &\mbox{ is $Y$--periodic}.&&\label{homeqf}
\end{align}
\end{mybox} 
\end{subequations}
\noindent{Relating back to our discussion in \cref{section_introduction}, the structure of Problem $P(\Omega)$ is motivated from the upscaled model of the microscopic problem \eqref{Eq:orgin} with $P(u^\es)=u^\es(1-u^\es)$ (cf. also \cite{raveendran2023homogenization,CIRILLO2016436}). Here} $u$ represents the particle concentration on macroscopic porous domain $\Omega$, while $W$, the solution to cell problems, relays the microscopic information, including the shape of the domain,
drift, and diffusion, to the macroscopic scale via the dispersion tensor, $D^*(W)$, given by
\begin{equation}\label{macrodiff}
     D^*(W):=\frac{1}{|Y|}\int_Y D(y)\left(I+\begin{bmatrix}
\frac{\partial w_1}{\partial y_1} & \frac{\partial w_2}{\partial y_1} \\
\frac{\partial w_1}{\partial y_2} & \frac{\partial w_2}{\partial y_2}
\end{bmatrix}\right)\di{y}
\end{equation}
where $I$ is the $2\times 2$ identity matrix and $|Y|$ is the Lebesgue measure of the set $Y$.
The periodic boundary condition \eqref{homeqf} and the Neumann influx boundary condition \eqref{cellbn3} are common in effective models derived by homogenization (see \cite{lukkassen2002two}).

We assume the reaction rate $f:\timeint\times\Omega\longrightarrow\mathbb{R}$, the initial condition \(g:\Omega\longrightarrow\mathbb{R}\), the microscopic diffusion matrix $D:Y\rightarrow \mathbb{R}^{2\times 2}$, the microscopic velocity field $B:Y\rightarrow \mathbb{R}^2$, and the nonlinear functions  $G_1,G_2:\mathbb{R}\rightarrow \mathbb{R}$ are given and satisfy additional assumptions discussed later.
Note that, to calculate the macroscopic equation and dispersion tensor, we need to, at each point, $(t,x)\in \timeint\times\Omega$, solve the microscopic problem in $Y$.

\paragraph{Assumptions.} From now on $C$ denotes a positive real number, possibly changing its value from line to line.
We consider the following restrictions on data and model parameters:
\begin{enumerate}[label=({A}{{\arabic*}})]
  \item\label{A1} 
The microscopic diffusion matrix satisfies $D\in (H^1_\#(Y)\cap L^\infty(Y))^{2\times2}$ and there exists $\theta>0$  such that 
 \begin{equation*}
 \theta |\eta |^{2} \leq
D\eta\cdot \eta \quad\text{for all}\  \eta\in \mathbb{R}^2 \; \mbox{and almost all } y\in Y;
 \end{equation*}
 
\item  $G_1,G_2:\mathbb{R}\rightarrow\mathbb{R}$ are locally Lipschitz functions, i.e., Lipschitz on compact sets;
 	 \label{A2}
  \item \label{A3} The microscopic drift velocity $B\in (H^1_\#(Y)\cap L^\infty(Y))^2$ satisfies
\begin{equation*}
    \begin{cases}
 	      \mathrm{div}_yB=0\hspace{.4cm}\mbox{in} \hspace{.4cm}   Y,\\
 	      B\cdot n_y =0 \hspace{.4cm}\mbox{on} \hspace{.4cm}  \Gamma_N;
 	\end{cases}
\end{equation*}

\item The reaction rate satisfies $f\in C^{\alpha,\frac{\alpha}{2}}((0,T)\times\Omega)$ and the initial condition $g\in C^{2+\alpha}(\Omega)$, for some $0<\alpha<1.$
\label{A4}
 \end{enumerate}
Assumptions \ref{A1}-\ref{A4} may seem technical but can be satisfied by many physically relevant functions.
For example, {Assumption \ref{A3}} can be satisfied by taking $B$ to be the solution to an incompressible Stokes problem with no penetration into the obstacle and periodic boundary conditions across $\partial Y\setminus \Gamma_N$.
 This is also the approach we have taken in our simulations, see \cref{section_simulation}.
 
    \par To discuss the concept of weak solution to Problem $P(\Omega)$, we introduce the following functional spaces:
\begin{align*}
H_{\#}^1(Y)&:=\{v\in H^1(Y): v \;\mbox{is}\; Y- \mbox{periodic}\},\\
    \mathcal{U}&:=\{v\in L^2(\timeint;H_0^{1}(\Omega)): \partial_t v\in L^2(\timeint;H^{-1}(\Omega))\},\\
    \text{and }\mathcal{W}&:=\left\{v\in H_{\#}^1(Y): \int_Yv(y)\di{y}=0\right\},
\end{align*}
where $H_{\#}^1(Y)$ is equipped with the standard $H^1(Y)$ norm.
We denote by $\langle\cdot,\cdot\rangle$ the duality pairing between $H^{-1}(\Omega)$ and $H_0^1(\Omega).$
Finally, we define the weak solution of Problem $P(\Omega)$ in the following sense:
\begin{definition}\label{D1}
We say that $(u,W)$ is a weak solution to Problem $P(\Omega)$ if $u\in\mathcal{U}$ with $u(0,\cdot) = g$ and, for almost every $(t,x)\in \timeint \times \Omega$, $w_i\in \mathcal{W}$ and the following integral equations are satisfied:
\begin{subequations}
\begin{align}
       \langle \partial_t u, \phi \rangle+\int_{\Omega} D^*(W) \nabla u \cdot\nabla \phi \di{x} &=  \int_{\Omega} f \phi \di{x},\label{wf}\\
       \int_{Y}\big(D\nabla_y w_i  - G_i(u(t,x))Bw_i\big)\cdot \nabla_y \psi\di{y}&=\int_{Y} \mathrm{div}_y( D e_i)\psi \di{y}-\int_{\Gamma_N} De_i\cdot n_y\psi \di{\sigma},\label{wfcell}
\end{align}
\end{subequations}
for all  $(\phi, \psi)\in H^1(\Omega)\times H_{\#}^1(Y)$ and $i\in\{1,2\}$.
\end{definition}

We remark that in Problem $P(\Omega)$ we can also have nonhomogeneous Dirichlet boundary conditions $u=h$ on $\partial\Omega$  for some given $h\in H^{\frac{1}{2}}(\partial\Omega).$ This regularity allows us to extend $h$ to $\Tilde{h}\in H^1(\Omega)$ and transform the original problem into
\begin{align*}
    \partial_t \Tilde{u} +\mathrm{div}( -D^*(W)\nabla_x  \Tilde{u})&=\Tilde{f}(W) &\mbox{in}& \hspace{.2cm}  \timeint\times \Omega, \\
    \Tilde{u}&=0 &\mbox{on}& \hspace{.3cm}  \timeint\times \partial\Omega,\\
    \Tilde{u}(0)&=\Tilde{g} &\mbox{in}& \hspace{.3cm}  \overline{\Omega},\\
    \vspace{1cm}\nonumber\\
    \mathrm{div}_y\left( -D \nabla_y w_i+G_i(\Tilde{u}+\Tilde{h}) B w_i\right)&=\mathrm{div}_y(D e_i) &\mbox{in}& \hspace{.2cm}  Y,\\
    \left( -D\nabla_y w_i+BG_i(\Tilde{u}+\Tilde{h})w_i\right)\cdot n_y &= \left( De_i\right)\cdot n_y &\mbox{on}& \hspace{.2cm} \Gamma_N,\\
    w_i\, &\mbox{ is $Y$--periodic},&&
\end{align*}
where $\Tilde{u}:=u-\Tilde{h},\;\Tilde{f}(W):=f-\mathrm{div}(D^*(W)\nabla \Tilde{h} )$ and $\Tilde{g}:=g-\Tilde{h}$.

\section{Weak solvability of Problem $P(\Omega)$}\label{section_analysis}
In this section, we prove the existence of solutions to Problem $P(\Omega)$ in the sense of Definition \ref{D1}.
There are two main difficulties in the analysis of the involved equations: (i) the nonlinear coupling in the drift, and (ii) the positivity of the dispersion tensor.

The main strategy of our proof follows four distinct steps:
\begin{itemize}[leftmargin=1.37cm]
    \item[Step 1:]  We construct a two-scale iterative scheme based on Problem $P(\Omega)$.
    This scheme is constructed by decoupling $P(\Omega)$ into a family of linear problems, which we refer to as $P^k(\Omega)$, $k\in\mathbb{N}$.
    The notion of weak solutions $(u^k,W^k)$ to Problem $P^k(\Omega)$ is given in \cref{D2}.
    \item[Step 2:] We construct an auxiliary problem similar to \eqref{cell3}--\eqref{homeqf} and show it is well-posed in \cref{L1}.
    This helps us to understand the coupling between $G_i(u^k)$ and $W^k$.
    We then show in Lemma \ref{L2} that the dispersion tensor is bounded and positive definite, uniformly with respect to the index $k$.
    \item[Step 3:] With Lemma \ref{L1} and Lemma \ref{L2}, we show that for each $k$ the problems $P^k(\Omega)$ are well-posed, their solutions satisfy $k$-independent energy estimates, and are also uniformly bounded; see both Theorem \ref{theorem_well_posedness} and Lemma \ref{L4}.
    \item[Step 4:] Finally, we prove our main result (Theorem  \ref{theorem_existance_of_p(omega)}) -- the solution to the iterative scheme converges in a suitable sense to a solution to Problem $P(\Omega)$.
\end{itemize}
\subsection{Iterative scheme}\label{IterationScheme}
We begin by introducing the iterative scheme and its corresponding weak formulation.
We set $u^0=g$, and, for any $k\in \mathbb{N}\cup\{0\}$, we denote as $u^{k+1}, w^{k}_1,$ and $w^{k}_2$ the solutions to the following decoupled system:
\begin{subequations}\label{Eq:IterationScheme}
\begin{mybox}{Diffusion-drift problem}
\vspace{-.3cm}
    \begin{align}
    \mathrm{div}_y\left( -D \nabla_y w_i^{k}+ G_i(u^{k})B w_i^{k}\right)&=\mathrm{div}_y(D e_i)&\mbox{in}& \hspace{.2cm}  Y,\label{kc1}\\
    \left( -D\nabla_y w_i^{k}+BG_i(u^k)w_i^{k}\right)\cdot n_y &= \left( De_i\right)\cdot n_y &\mbox{on}& \hspace{.2cm} \Gamma_N,\label{kc2}\\
    w_i^{k}\, &\mbox{ is $Y$--periodic},&i\in &\{1,2\}\label{kc4}
\end{align}
\end{mybox}
\begin{mybox}{Reaction-dispersion problem}
\vspace{-.3cm}
\begin{align}
    \partial_t u^{k+1} +\mathrm{div}( -D^*(W^k)\nabla_x  u^{k+1})&=f &\mbox{in}& \hspace{.2cm}  \timeint\times \Omega, \label{kp1}\\
    u^{k+1}(0)&=g &\mbox{in}& \hspace{.3cm}  \overline{\Omega},\label{kp2}\\
    u^{k+1}&=0 &\mbox{on}& \hspace{.3cm}  \timeint\times \partial\Omega,\label{kp3}
    \end{align}
\end{mybox}
\noindent where the dispersion tensor $D^*(W^k)$ is given by
\begin{mybox}{Effective dispersion tensor}
\begin{equation}\label{kmacrodiff}
     D^*(W^k):=\frac{1}{|Y|}\int_Y D(y)\left(I+\begin{bmatrix}
\frac{\partial w^{k}_1}{\partial y_1} & \frac{\partial w^{k}_2}{\partial y_1} \\
\frac{\partial w^{k}_1}{\partial y_2} & \frac{\partial w^{k}_2}{\partial y_2}
\end{bmatrix}\right)\di{y}.
\end{equation}
\end{mybox}
\end{subequations}
We refer to the iterative scheme \eqref{kc1}--\eqref{kp3} together with \eqref{kmacrodiff}, as Problem $P^k(\Omega)$. 
We define the concept of weak solutions to Problem $P^k(\Omega)$ in a similar manner as for the Problem $P(\Omega)$ (see \cref{D1}).
\begin{definition}\label{D2}
   Given $u^k\in\mathcal{U},$  we say that $(u^{k+1},W^k)$ is a weak solution to Problem $P^k(\Omega)$ if $u^{k+1} \in \mathcal{U}$ with $u^{k+1}(0,\cdot) = g$ and, for almost every $(t,x) \in \timeint\times \Omega$, $w_i^{k} \in \mathcal{W}$ and the following integral equations are satisfied:
\begin{subequations}
\begin{align}
       \langle \partial_t u^{k+1}, \phi \rangle+\int_{\Omega} D^*(W^{k}) \nabla u^{k+1} \cdot\nabla \phi \di{x} &=  \int_{\Omega} f \phi \di{x}\label{wfs1}\\
       \int_{Y}\left(D\nabla_y w^{k}_i-G_i(u^{k}(t,x))Bw^{k}_i\right)\cdot \nabla_y \psi\di{y}&=\int_{Y} \mathrm{div}_y( D e_i)\psi \di{y}-\int_{\Gamma_N} De_i\cdot n_y\psi \di{\sigma},
\end{align}
\end{subequations}
for all  $(\phi, \psi)\in H^1(\Omega)\times H_{\#}^1(Y)$ and $i\in\{1,2\}$.
\end{definition}
\subsection{Existence results for the Problem $P^k(\Omega)$}
In this section, we study the well-posedness of the iterative scheme \eqref{kc1}--\eqref{kmacrodiff}.
To analyze the scheme, we first focus on the elliptic equation \eqref{kc1}--\eqref{kc4} by constructing an auxiliary problem similar to \eqref{cell3}--\eqref{homeqf}.
We show this auxiliary problem is well-posed, prove energy estimates, and establish the regularity of weak solutions.
Then with the help of Lemma \ref{L2}, which establishes the uniform positive definite property and boundedness of $D^*(W^k)$, we show our iterative scheme Problem $P^k(\Omega)$ is well-posed. {Note that if $G_i(u_k),$ is a constant, then the positivity of the dispersion tensor follows via standard arguments, and further if $G_i(u_k)=0$ then we can show that the dispersion tensor is symmetric. The main difficulty of showing $D^*(W^k)$ is positive definite is that the functions $w_i^k$ depend nonlinearly on $u_k$. This nonlinear dependence complicates the analysis, as it introduces intricate interactions that must be carefully addressed to establish the positive definiteness of the tensor }

\begin{lemma}\label{L1}
Let $p \in \mathbb{R}$. Consider the following auxiliary problem: Find $W(p,\cdot)=(w_1(p,\cdot),w_2(p,\cdot)) \in \mathcal{W}^2$ satisfying
\begin{subequations}\label{system:cellproblem_with_p}
\begin{align}
    \mathrm{div}_y\left(- D \nabla_y w_i+pB w_i\right)&=\mathrm{div}_y(D e_i) &\mbox{in}& \hspace{.2cm}  Y,\label{aux1}\\
    \left( -D\nabla_y w_i+pB w_i\right)\cdot n_y &= \left( De_i\right)\cdot n_y &\mbox{on}& \hspace{.2cm} \Gamma_N,\label{aux2}\\
    w_i\, &\mbox{ is $Y$--periodic},&&\label{aux3}
    \end{align}
\end{subequations}
where $i\in\{1,2\}$.
{Assume \ref{A1}--\ref{A3}} holds, then 
\begin{enumerate}[label=(\roman*)]
    \item there exists a unique weak solution $w_i(p,\cdot)\in \mathcal{W}$   to the Problem \eqref{aux1}--\eqref{aux3},
 \item there exists a constant $C>0$ independent of $p$ such that
\begin{equation}
   \|\nabla_y w_i(p,\cdot)\|_{L^2(Y)}\leq C,\label{L1e1}
\end{equation}
\item there exists a constant $C>0$ such that, for all $p_1,p_2\in \mathbb{R}$, 
\begin{equation}\label{L1e2}
    \int_Y |\nabla_y( w_i(p_1,y)-w_i(p_2,y))|^2 \di{y}\leq C |p_1-p_2|^2.
\end{equation}
\end{enumerate}

\end{lemma}
\begin{proof}
\par \textbf{\textit{(i)}} We begin by stating the weak form of \eqref{aux1}--\eqref{aux3}:
\begin{equation}
      \int_{Y} \big(D\nabla_y w_i- pBw_i\big)\cdot \nabla_y \psi\di{y}=\int_{Y} \mathrm{div}_y( D e_i)\psi \di{y}-\int_{\Gamma_N} De_i\cdot n_y\psi \di{\sigma}\quad \label{auxwf}
\end{equation}
for all $\psi\in H_{\#}^1(Y)$ and $i\in\{1,2\}$. 
From \ref{A1}, we have \begin{equation}\label{L1e3}
    \int_Y \mathrm{div}_y (D(y) e_i) \di{y}= \int_{\Gamma_N} (D(y) e_i)\cdot n_y \di{\sigma}.
\end{equation}
Using identity \eqref{L1e3} together with standard arguments involving the classical Fredholm alternative, we have the existence and uniqueness of $w_1(p,\cdot)$ and $w_2(p,\cdot)$ in the sense of \eqref{auxwf} (for details we refer the reader \cite[Section 4]{ALLAIRE20102292}).
\par\textbf{\textit{(ii)}} First, using integration by parts,  we have the following equation:
\begin{align}
    \int_{Y}pBw_i\cdot \nabla_y w_i\di{y}&=\frac{p}{2}\int_{Y}B\nabla_y (w_i^2)\di{y}\nonumber\\
    &=-\frac{p}{2}\int_{Y}w_i^2\mathrm{div}_y B\di{y}+\frac{p}{2}\int_{\partial Y\backslash \Gamma_N} B\cdot n_y w_i^2\di{\sigma}\nonumber\\
    &\hspace{2cm}+\frac{p}{2}\int_{\Gamma_N} B\cdot n_y w_i^2\di{\sigma}.\label{xL1e4}
\end{align}
 By \ref{A3}, we have that $\int_{Y}w_i^2\mathrm{div}_y B\di{y}=0$ and $\int_{\Gamma_N} B\cdot n_y w_i^2\di{\sigma}=0$.
 Additionally, by the periodicity of $w_i$ and $B$, we have $\int_{\partial Y\backslash \Gamma_N} B\cdot n_y w_i^2\di{\sigma}=0$, and so we obtain
 \begin{equation}
    \int_{Y}pBw_i\cdot \nabla_y w_i\di{y}=0.\label{L1e4}
\end{equation}
Now, taking the test function $\psi = w_i$ in the weak formulation \eqref{auxwf} and making use of \eqref{L1e4}, we obtain
\begin{equation}
     \int_{Y} D\nabla_y w_i\cdot\nabla_y w_i \di{y} =\int_{Y} \mathrm{div}_y (D e_i) w_i \di{y}-\int_{\Gamma_N}\left( De_i\right)\cdot n_y w_i \di{\sigma}.\label{L1e5}
\end{equation}
Making use of \ref{A1} and integration by parts, we see that 
\begin{equation}
     \theta \int_{Y} |\nabla_y w_i|^2  \di{y} \leq \int_{Y} | (D e_i) \nabla_y w_i| \di{y}.\label{L1e6}
\end{equation}
Young's inequality applied to the right-hand side of \eqref{L1e6} together with \ref{A1}, yields \eqref{L1e1}.
\par \textbf{\textit{(iii)}} Let
\begin{equation*}
    \overline{w_i}:=w_i(p_1,y)-w_i(p_2,y) \text{ and }\overline{p}:=p_1-p_2
\end{equation*}
and consider the weak formulations \eqref{auxwf} with $p= p_1,\,p_2$. 
We take the test function $\psi = \overline{w_i}$ in both weak forms and subtract the equations to arrive at
\begin{equation}
    \int_{Y}\big(D(y)\nabla_y \overline{w_i}- B(y)(p_1w_i(p_1,y)-p_2w_i(p_2,y)))\big)\cdot \nabla _y \overline{w_i}\di{y}=0\label{L1e7}.
\end{equation}
Adding and subtracting the term $\int_Y B(y)p_2w_i(p_1,y)\di{y}$ to \eqref{L1e7} and making use of assumption \ref{A1}, we obtain
\begin{equation}\label{L1e9}
    \theta\int_{Y} |\nabla_y \overline{w_i} |^2 \di{y}
    \leq \int_{Y}B(y)\big(\overline{p}w_i(p_1,y)+p_2\overline{w_i}\big)\cdot \nabla _y \overline{w_i} \di{y}.
\end{equation}
Since $w_1$ and $w_2$ are periodic, $\overline{w_i}$ is periodic as well.
Therefore, we may follow similar arguments which lead to  \eqref{L1e4}  to obtain
\begin{equation}\label{L1e10}
    \int_{Y}B(y)p_2\overline{w_i}\cdot \nabla _y \overline{w_i}\di{y}=0.
\end{equation}
Returning to \eqref{L1e9}, using Young's inequality, we have
\begin{align}
    \theta\int_{Y} |\nabla_y \overline{w_i} |^2 \di{y}\leq\int_{Y}B(y)\overline{p}w_i(p_1,y)\cdot \nabla _y \overline{w_i}\di{y}&\leq C \overline{p}\int_{Y}|w_i(p_1,y)|| \nabla _y \overline{w_i}| \nonumber\\
    &\leq \frac{1}{2\theta}|\overline{p}|^2 \|w_i(p_1)\|_{L^2(Y)}^2+\frac{\theta}{2} \int_{Y} |\nabla_y \overline{w_i} |^2 \di{y}.\label{L1e11}
\end{align}
Then, with \eqref{L1e1}, we arrive at
\begin{equation}
    \int_{Y} |\nabla_y \overline{w_i} |^2 \di{y}\leq  C|\overline{p}|^2.\label{L1e12}
\end{equation}

\end{proof}
\par To study the weak solvability of \eqref{kp1}--\eqref{kp3}, we first show that the dispersion tensor \eqref{kmacrodiff} is both uniformly positive definite and uniformly bounded.
To this end, for any $p\in\mathbb{R}$, we define $\overline{D}(p)\in \mathbb{R}^{2\times 2}$ by
\begin{align}
    [\overline{D}(p)]_{i,j}:=\dfrac{1}{|Y|}\int_Y D(y)\left(e_j+\nabla_y w_j (p,y)\right) \cdot e_i\di{y},\label{dbar}
\end{align}
where $i,j\in \{1,2\}$ and $w_j(p,y)$ is the weak solution to \eqref{aux1}--\eqref{aux3}.  
We begin by showing that the matrix $\overline{D}(p)$ is Lipschitz continuous with respect to $p$.
\begin{lemma}\label{L3}
    {Assume \ref{A1}--\ref{A3}} hold.
    There exists a constant $C>0$, independent of $p$ and $q$, such that 
 \begin{equation*}
     |[\overline{D}(p)]_{i,j}-[\overline{D}(q)]_{i,j}|\leq C |p-q|.
 \end{equation*}
\end{lemma}
\begin{proof}
     From definition \eqref{dbar}, assumption \ref{A1}, and Lemma \ref{L1} $(iii)$, we have
\begin{align}
     |[\overline{D}(p)]_{i,j}-[\overline{D}(q)]_{i,j}|&\leq C \int_Y |D(y)|\left|\nabla_y( w_i(p,y)-w_i(q,y))\right|\di{y}\nonumber\\
     &\leq C \left(\int_Y |\nabla_y( w_i(p,y)-w_i(q,y))|^2 \di{y}\right)^{\frac{1}{2}}\nonumber\\
     &\leq  C|p-q|.\label{p1}
 \end{align}
\end{proof}

We now show that the dispersion tensor is both uniformly bounded and uniformly positive definite. 
\begin{lemma}\label{L2}
{Assume \ref{A1}--\ref{A3}} hold and let $M>0$.
For $p=(p_1,p_2)\in [-M,M]^2$, let $W_p:=(w_1(p_1,\cdot ),w_2(p_2,\cdot ))$ be the corresponding solutions of the auxiliary problem \eqref{aux1}--\eqref{aux3}. Then the macroscopic dispersion tensor $D^*(W_p)$ satisfies the following properties:
\begin{enumerate}[label=(\roman*)]
    \item There exists $\theta_M>0$ independent of $p$ such that 
         \begin{equation}\label{L2e1}
             \theta_M|\eta |^{2} \leq
             D^*(W_p)\eta \cdot \eta\quad \text{for all}\ \eta\in \mathbb{R}^2;
          \end{equation}

    \item There exist a constant $C>0$ independent of $p$ such that
         \begin{equation}\label{L2e2}
            | [D^*(W_p)]_{i,j}|\leq C,\quad (i,j\in\{1,2\}).
         \end{equation}
\end{enumerate}
\end{lemma}
\begin{proof}

    \par \textbf{\textit{(i)}} From the definition of $\overline{D}(\cdot)$ \eqref{dbar} and the definition of $D^*(\cdot)$ \eqref{macrodiff}, we have
    \begin{equation}
       [ D^*(W_p)]_{i,j}=[\overline{D}(p_i)]_{i,j},\label{Q}
    \end{equation}
    and hence,
\begin{equation}\label{xL2e1}
    [D^*(W_p)]_{i,j}=\frac{1}{|Y|}\int_Y D(y)\left(e_j+\nabla_y w_j(p_i,y)\right)\cdot e_i \di{y}.
\end{equation}
We consider the weak formulation for the Problem \eqref{aux1}--\eqref{aux3} (i.e., Equation \eqref{auxwf}) for index $j$, where we choose the test function $\psi=w_i$ and divide by $|Y|$:
\begin{equation}\label{L2e5}
    \frac{1}{|Y|}\int_{Y}\big(D\nabla_y w_j- p_jBw_j\big)\cdot \nabla_y w_i\di{y}
    =\frac{1}{|Y|}\int_Y\text{div}_y(D e_j)w_i\di{y}-\frac{1}{|Y|}\int_{\Gamma_N}\left( De_j\right)\cdot n_y w_i\di{\sigma}.
\end{equation}
Using integration by parts on $\int_Y\text{div}_y(D e_j)w_i\di{y}$ and noting the periodicity of $D(\cdot)$ and $w_i(\cdot)$, we arrive at
\begin{equation}\label{L2e6}
    \frac{1}{|Y|}\int_{Y}\big( D\nabla_y w_j- p_jBw_j+D e_j\big)\cdot \nabla_yw_i \di{y} =0.
\end{equation}
Adding \eqref{L2e6} to \eqref{xL2e1}, we see that
\begin{equation}\label{L2e7}
    [D^*(W_p)]_{i,j}=[A(p)]_{i,j}+[J(p)]_{i,j},
\end{equation}
where 
\begin{align}
    [A(p)]_{i,j}&:=\frac{1}{|Y|}\int_Y D(e_j+\nabla_yw_j)\cdot 
    (e_i+\nabla_y w_i)\di{y},\label{L2e8}\\
   \text{and } [J(p)]_{i,j}&:=-\frac{1}{|Y|}\int_{Y}p_jBw_j\cdot \nabla_y w_i\di{y}.\label{L2e9}
\end{align}
Using \ref{A3} and integration by parts on $\int_{Y}p_jBw_j\cdot \nabla_y w_i\di{y}$, we obtain
\begin{align}
    -[J(p)]_{i,j}&=-\frac{p_j}{|Y|}\int_{Y}\nabla_y(Bw_j)  w_i\di{y}+\frac{p_j}{|Y|}\int_{\partial Y}B\cdot n_y w_iw_j\di{\sigma}\nonumber\\
    &=-\frac{1}{|Y|}\int_{Y}p_jBw_i\cdot \nabla_y w_j\di{y}\nonumber\\
    &=[J(p)]_{j,i}.\label{L2e10}
\end{align}
Combining \eqref{L2e8} and \eqref{L2e10} shows that $A(p)$ and $J(p)$ are the symmetric and skew-symmetric parts of $D^*(W_p)$.
Since $J(p)$ is a $2\times 2$ skew-symmetric matrix, we have 
\begin{equation}\label{L2e11}
    J(p) \eta\cdot \eta=0\quad\text{for all}\ \eta\in \mathbb{R}^2.
\end{equation}
So, to prove \eqref{L2e1} it is enough to establish that $A(p)$ is in fact positive definite {for all $p\in [-M,M]^2$}. {For that by the standard arguments we first show that for a fixed $p$,  $A(p)$ is positive definite. }
Since $A(p)$ is symmetric, using the structure \eqref{L2e8} we have for all $\eta=(\eta_1,\eta_2)\in \mathbb{R}^2$, 
\begin{align}
A(p)\eta\cdot\eta&=\eta_1^2\frac{1}{|Y|}\int_Y D(e_1+\nabla_yw_1)\cdot 
    (e_1+\nabla_y w_1)\di{y}\nonumber\\
    &\hspace{1cm}+2\eta_1 \eta_2\frac{1}{|Y|}\int_Y D(e_1+\nabla_yw_1)\cdot 
    (e_2+\nabla_y w_2)\di{y}\nonumber\\
    &\hspace{1.5cm}+\eta_2^2 \frac{1}{|Y|}\int_Y D(e_2+\nabla_yw_2)\cdot 
    (e_2+\nabla_y w_2)\di{y}.\nonumber
    \end{align}
    Distributing $\eta_1$ and $\eta_2$ inside the integrals, we get
    \begin{align}
    A(p)\eta\cdot\eta&=\frac{1}{|Y|}\int_Y D(\begin{bmatrix} \eta_1 \\ 0  \end{bmatrix}+\eta_1\nabla_yw_1)\cdot 
    (\begin{bmatrix} \eta_1 \\ 0  \end{bmatrix}+\eta_1\nabla_y w_1)\di{y}\nonumber\\
    &\hspace{1cm}+2\frac{1}{|Y|}\int_Y D(\begin{bmatrix} \eta_1 \\ 0  \end{bmatrix}+\eta_1\nabla_yw_1)\cdot 
    (\begin{bmatrix} 0 \\ \eta_2  \end{bmatrix}+\eta_2\nabla_y w_2)\di{y}\nonumber\\
    &\hspace{1.5cm}+ \frac{1}{|Y|}\int_Y D(\begin{bmatrix} 0 \\ \eta_2  \end{bmatrix}+\eta_2\nabla_yw_2)\cdot 
    (\begin{bmatrix} 0 \\ \eta_2  \end{bmatrix}+\eta_2\nabla_y w_2)\di{y}\label{xL2e2}.
      \end{align}
      Rearranging the right side of the identity \eqref{xL2e2} and using \ref{A1} yields
    \begin{align}
           A(p)\eta\cdot\eta&=\frac{1}{|Y|}\int_Y D(\begin{bmatrix} \eta_1 \\ \eta_2  \end{bmatrix}+\sum_{i=1}^2\eta_i\nabla_yw_i)\cdot 
    (\begin{bmatrix} \eta_1 \\ \eta_2  \end{bmatrix}+\sum_{i=1}^2\eta_i\nabla_yw_i)\di{y}\nonumber\\
    &\geq \theta\frac{1}{|Y|}\int_Y |\eta+ \sum_{i=1}^2\eta_i\nabla_yw_i)|^2\;\di{y}\nonumber\\
    &\geq 0.\label{L2e12}
\end{align}
If $A(p)\eta\cdot\eta=0$ for some $\eta\in \mathbb{R}^2$, then we have
\begin{align*}
|\eta+ \sum_{i=1}^2\eta_i\nabla_yw_i)|^2=0 \,\,\,\mbox{for almost every }\, y\in Y.
\end{align*}
Consequently, this yields
\begin{equation}
\sum_{i=1}^2\eta_iw_i=C-\eta\cdot y,\label{L2e13}
\end{equation}
for some constant $C $.
Since the $w_i$ are periodic, identity \eqref{L2e13} is only possible for $\eta=0$.
Hence, $A(p)\eta\cdot \eta=0$ only for $\eta= 0$.
Therefore, combining this with \eqref{L2e11}, we get
\begin{equation}\label{L2e15}
    D^*(W_p)\eta\cdot \eta> 0.
\end{equation}
for all $p\in [-M,M]^2$ and all $\eta\neq0$.
 Define the function $F:[-M,M]^2\times \partial B(0,1)\rightarrow \mathbb{R}$ as 
\begin{equation*}
    F(p,\xi):= D^*(W_p)\xi\cdot \xi \,\,\,\mbox{for all}\; \xi\in \partial B(0,1).
\end{equation*}
     We have from Lemma \ref{L3} that  the function $p_i\mapsto \overline{D}(p_i)$ is continuous.
     Using the definition \eqref{Q}, we see that the function $p\mapsto D^*(W_p)$ is also continuous.
     Therefore, $F$ is continuous and $F(p,\xi)>0$ over the compact set $ ([-M,M]^2\times \partial B(0,1))$. Hence, there exist a $\theta_M>0$ such that
\begin{equation*}
    F(p,\xi)>\theta_M
\end{equation*}
for all $(p,\xi)\in  ([-M,M]^2\times \partial B(0,1))$.
So, for any $\eta\neq 0$ we have 
\begin{align*}
    D^*(W_p)\frac{\eta}{|\eta|}\cdot \frac{\eta}{|\eta|}&>\theta_M,\\
     \text{i.e., } D^*(W_p)\eta\cdot \eta&>\theta_M|\eta|^2.
\end{align*}
\textbf{\textit{(ii)}} The inequality \eqref{L2e2} follows directly from \ref{A1}, $(ii)$ of Lemma \ref{L1}, and the identity 
\begin{equation*}
    [D^*(W_p)]_{i,j}=\frac{1}{|Y|}\int_Y D\left(e_j+\nabla_y w_j\right)\cdot e_i\di{y}
\end{equation*}together with the Cauchy--Schwartz inequality.
\end{proof}

\begin{theorem}[Well-posedness of $P^k(\Omega)$]
\label{theorem_well_posedness}
   Assume \ref{A1}--\ref{A4} hold and let $u^0 = g$. Then, there exist a sequence $(u^{k},W^k)_{k\in \mathbb{N}\cup\{0\}}\subset\mathcal{U}\times L^\infty(\timeint\times\Omega;\mathcal{W}^2)$ such that, for each $k$, $(u^{k+1},W^k)$ uniquely solves the iterative scheme given by \eqref{kc1}--\eqref{kmacrodiff} in the sense of \cref{D2}.
  Moreover, we have 
\begin{align}
  \|u^{k}\|_{L^{\infty}(\timeint\times\Omega)}&\leq \|g\|_{L^\infty(\Omega)}+T\|f\|_{L^{\infty}(\timeint\times\Omega)}\label{T1e1}
\end{align}
   for all $k\in \mathbb{N}$.
\end{theorem}
\begin{proof}
We prove the existence of solutions via induction.
By construction, we have $u^0=g$ and the linear elliptic problems \eqref{kc1}--\eqref{kc4} have unique solutions $w_i^0(p,\cdot)\in \mathcal{W}$, where $p=G_i(u^0(t,x))$ for almost all $(t,x)\in \timeint\times\Omega$ via \cref{L1} (i). 
Owing to \cref{L1} (ii), we have that $\|\nabla_y w_i(p,\cdot)\|_{L^2(\Omega)}\leq C$ uniformly in $p$.
This implies $\|w_i(p,\cdot)\|_{\mathcal{W}}\leq C$ via the Poincaré–Wirtinger inequality and, as a consequence, it also holds that $W^0=(w_1^0,w_2^0)\in L^\infty(\timeint\times\Omega;\mathcal{W}^2)$.
\par Let
\begin{equation*}
    m:=\|g\|_{L^\infty(\Omega)}+T\|f\|_{L^{\infty}(\timeint\times\Omega)}
\end{equation*}
and set
\begin{equation*}
    M:=\max_{r\in[-m,m]}\max_{i\in\{1,2\}} |G_i(r)|.
\end{equation*}
$M$ is well-defined as the functions $G_i$ are continuous.
Since $u^0=g$, we also have $|G_i(u^0)|\leq M$.
Using Lemma \ref{L2}, we see that the dispersion tensor $D^*(W^0)$ satisfies
\begin{equation*}
    D^*(W^0)\eta \cdot \eta\geq\theta_M|\eta|^2
\end{equation*}
 for all $\eta \in \mathbb{R}^2$ almost everywhere in $\timeint\times\Omega$, where the constant $\theta_M>0$ depends only on the choice of $M$.

The Problem \eqref{kp1}--\eqref{kp3} is a standard linear parabolic problem.
Benefiting of Assumptions \ref{A1}--\ref{A4} and of the uniform positivity of $D^*(W^0)$, the existence of a unique weak solution  $u^{1}\in \mathcal{U}$ to the Problem \eqref{kp1}--\eqref{kp3} with $k=0$ can be shown via standard Galerkin approximation arguments (see, e.g., \cite[Chapter 7, Theorem 3]{evans2010partial}). The estimate 
 \begin{equation}\label{u1globalbound}
  \|u^{1}\|_{L^{\infty}(\timeint\times\Omega)}\leq \|g\|_{L^\infty(\Omega)}+T\|f\|_{L^{\infty}(\timeint\times\Omega)},
\end{equation}
follows by an application of Duhamel's principle \cite[Chapter 5]{precup2012linear}, for details see \cite[Lemma 10]{EDEN2022103408}.
\par The solvability of the Problem \eqref{kc1}--\eqref{kp3} for any  $k\in \mathbb{N}$  follows by induction. Indeed, assume $u^k\in\mathcal{U}$ satisfies $u^k(0,\cdot)=g$ and
\begin{equation}\label{xT1e1}
    \|u^k\|_{L^{\infty}(\timeint\times\Omega)}\leq m,
\end{equation}
it implies
\begin{equation}\label{xT1e2}
    G_i(u^k(t,x)) \in [-M,M], \quad i=1,2.
\end{equation}
Then the existence of $W^k\in L^\infty(\timeint\times\Omega;\mathcal{W}^2)$ solving \eqref{kc1}--\eqref{kc4} follows by Lemma \ref{L1} in the same way as for $W^0$.
Relying on Lemma \eqref{L2} and \eqref{xT1e2}, we obtain that $D^*(W^k)$ is bounded and that for all $\eta \in \mathbb{R}^2$ it satisfies
\begin{equation*}
    D^*(W^k)\eta \cdot \eta\geq\theta_M|\eta|^2
\end{equation*}
 almost everywhere in $\timeint\times\Omega$, where $\theta_M>0$ does not depend on $k$.
 Following the same arguments as before, we get the existence and uniqueness of $u^{k+1}\in\mathcal{U}$. 
 The estimate 
 \begin{equation*}
  \|u^{k+1}\|_{L^{\infty}(\timeint\times\Omega)}\leq \|g\|_{L^\infty(\Omega)}+T\|f\|_{L^{\infty}(\timeint\times\Omega)},
\end{equation*}
follows similarly to \eqref{u1globalbound}.

\end{proof}

\subsection{Convergence of the iterative scheme}
Now, we show that the sequence $(u^k,W^k)_{k\in \mathbb{N}\cup \{0\}}$ satisfies uniform energy estimates.
These estimates are crucial in establishing the convergence of our scheme. 
We also present a proposition that will help us pass to the limit in the weak formulation.
\begin{lemma}\label{L4}
  Assume \ref{A1}--\ref{A4} hold and 
  let $u^{k}\in \mathcal{U}$ be weak solution to the iterative scheme \eqref{kc1}--\eqref{kp3}.
  Then, there exists a constant $C>0$ independent of $k$ such that
 \begin{subequations}
    \begin{align}
    \|u^{k}\|_{L^{\infty}(\timeint;L^2(\Omega))}&\leq C,\label{T2e1}\\ 
    \|\nabla u^{k}\|_{L^{2}(\timeint;L^2(\Omega))}&\leq C,\label{T2e2}\\
    \|\partial_t u^{k}\|_{L^{2}(\timeint;H^{-1}(\Omega))}&\leq C\label{T2e3}.
    \end{align}
  \end{subequations}
\end{lemma}

\begin{proof}
Using $(i)$ of Lemma \ref{L2} and \eqref{T1e1}, there exist  $\theta_M>0$ independent of $k$, depending on $M$ such that 
\begin{equation*}
    D^*(W^k)\eta \cdot \eta\geq\theta_M|\eta|^2
\end{equation*}
 for all $\eta \in \mathbb{R}^2$. 
 Estimates \eqref{T2e1}--\eqref{T2e3} then follow from the arguments made in  \cite[Chapter 7, Theorem 2]{evans2010partial}. 
\end{proof}
\begin{lemma}\label{Lemma:c2holder bound for uk}
     Assume \ref{A1}--\ref{A4} hold and 
  let $u^{k}\in \mathcal{U}$ be weak solution to the iterative scheme \eqref{kc1}--\eqref{kp3}.
  Then, there exists a constant $C>0$ independent of $k$ and $0<\gamma<1$  such that
  \begin{equation}\label{eq:c2holderbound}
      |u^k|_{2+\gamma}\leq C,
  \end{equation}
  where $|\cdot|_{2+\gamma}$ is the Hölder norm (see \cite[Chapter 4]{lieberman1996second} for the detailed definition of the norm).
\end{lemma}
\begin{proof}
Using the main theorem from \cite{naumann2005global} (see also \cite{naumann2005interior}) and \eqref{T2e2}, we get that there exist a $p>2$, such that
\begin{equation}\label{eq:higher integrability}
     \|\nabla u^{k}\|_{L^{2}(\timeint;L^p(\Omega))}\leq C,
\end{equation}
where the constant $C>0$ independent of $k$.
Since $\Omega\subset \mathbb{R}^2$ and $p>2$, from Morrey's inequality (see \cite[Theorem 9.12]{brezis2011functional}) and \eqref{eq:higher integrability}, we obtain 
 \begin{equation}\label{eq:holder continuity for uk}
     |u^k|_{\gamma}\leq C,
 \end{equation}
 where $C$ is independent of $k$ and $|\cdot|_{\gamma}$ is the Hölder norm with coefficient $\gamma:=1-\frac{2}{p}$.
Using Lemma \ref{L3}, and \ref{A2},  we obtain
\begin{align}
    |[D^*(W^{k})(t,x_1) -D^*(W^{k})(t,x_2)]_{i,j}|\leq C|u^k(t,x_1)-u^{k}(t,x_2)|,
\end{align}
for $i,j\in\{1,2\}$.
From the aforementioned estimates, we get
\begin{equation}\label{eq:holder estimate for dstar}
     |[D^*(W^k)(t,x)]_{i,j}|_{\gamma}\leq C,
\end{equation}
 where $C$ independent of $k$. Now, using the Schauder regularity theorem \cite[Theorem 5.14]{lieberman1996second} (see also \cite{lunardi1998schauder}) together with \ref{A4} and \eqref{eq:holder estimate for dstar} for \eqref{wfs1}, we obtain \eqref{eq:c2holderbound}.
\end{proof}
Note that from Lemma \ref{Lemma:c2holder bound for uk} we have the following inequality
\begin{equation}\label{eq:gradient global bound}
    \|\nabla u^k\|_{L^\infty((0,T)\times\Omega)}\leq C,
\end{equation}
where $C>0$ independent of $k$.
\begin{proposition}\label{lemma:productofsequence}
    Let $\xi_k,\psi_k,\xi,\psi\in L^2(\Omega)$ for all $k\in \mathbb{N}$. Assume 
    \begin{alignat}{4}
        \xi_k&\rightarrow\xi\hspace{1cm}&& \mbox{strongly in} \hspace{1cm}&&L^2(\Omega),\label{Eq:xi_n convergence}\\
        \psi_k&\rightharpoonup\psi&& \mbox{weakly in} &&L^2(\Omega)\label{Eq:psi_n convergence}
    \end{alignat}
    and there exist a $C>0$ independent of $k$, such that 
    \begin{equation}\label{Eq:xi_k estimate}
        \|\xi_k\|_{L^{\infty}(\Omega)}\leq C,
    \end{equation}
    then for any $\phi\in L^2(\Omega)$, we have
    \begin{equation}\label{eq:xi_n psi_n weak convergence l2}
    \lim_{k\rightarrow\infty}\int_{\Omega}\xi_k\psi_k\phi \di{x}\rightarrow \int_{\Omega}\xi\psi\phi \di{x}.
    \end{equation}
\end{proposition}
\begin{proof}
    Using \eqref{Eq:xi_n convergence} and \eqref{Eq:psi_n convergence}, we get 
    \begin{equation*}
        \xi_k\psi_k\;\rightharpoonup\;\xi\psi \;\;\;\mbox{weakly in}\;\; L^1(\Omega).
    \end{equation*}
    Hence for all $\phi\in C_c^\infty({\Omega})$, we have
     \begin{equation}\label{eq:weak l1 convergence}    \lim_{k\rightarrow\infty}\int_{\Omega}\xi_k\psi_k\phi \di{x}\rightarrow \int_{\Omega}\xi\psi\phi \di{x}.
    \end{equation}
    Now, from \eqref{Eq:psi_n convergence} and \eqref{Eq:xi_k estimate}, there exist a constant $C>0$ independent  of $k$, such that
    \begin{equation}
        \|\xi_k\psi_k\|_{L^2(\Omega)}\leq C.
    \end{equation}
    So, by the weak compactness theorem, we get $\xi_k\psi_k$ weakly converges to some $h\in L^2(\Omega)$ in $L^2(\Omega)$. Now, combining with \eqref{eq:weak l1 convergence}, we get $h=\xi\psi$ for a.e. $x\in \Omega$, which give us the required result \eqref{eq:xi_n psi_n weak convergence l2}.
\end{proof}
\begin{lemma}\label{lemma:cauchy}
     Assume \ref{A1}--\ref{A4} hold.    
  Let $u^{k}\in \mathcal{U}$ be the weak solution to the iterative scheme \eqref{kc1}--\eqref{kp3}, then $u^k $ is a Cauchy sequence in  $L^2(0,T;L^2(\Omega))$.
\end{lemma}
\begin{proof}
   Let $k\in\mathbb{N}$, we have
    \begin{align}
       \langle \partial_t u^{k+1}, \phi \rangle+\int_{\Omega} D^*(W^{k}) \nabla u^{k+1} \cdot\nabla \phi \di{x} &=  \int_{\Omega} f \phi \di{x}\label{eq:k+1 step}\\
       \langle \partial_t u^{k}, \phi \rangle+\int_{\Omega} D^*(W^{k-1}) \nabla u^{k} \cdot\nabla \phi \di{x} &=  \int_{\Omega} f \phi \di{x}\label{eq:k step},
\end{align}
for all $\phi\in H_0^1(\Omega)$.
Now, subtracting \eqref{eq:k step} from \eqref{eq:k+1 step} and choosing $\phi=(u^{k+1}-u^{k})$, we obtain
\begin{align}
    \frac{1}{2}\frac{d}{dt} \|u^{k+1}-u^{k}\|_{L^2(\Omega)}^2+ \int_{\Omega} (D^*(W^{k}) \nabla u^{k+1}-D^*(W^{k-1}) \nabla u^{k}\cdot\nabla (u^{k+1}-u^{k}) \di{x}=0.
\end{align}
Adding and subtracting $\int_{\Omega} D^*(W^{k}) \nabla u^{k}\cdot\nabla (u^{k+1}-u^{k}) \di{x}$, we get
\begin{align}
    \frac{1}{2}\frac{d}{dt} \|u^{k+1}-u^{k}\|_{L^2(\Omega)}^2+\theta \|\nabla u^{k+1}-\nabla u^{k}\|_{L^2(\Omega)}^2 \leq \int_{\Omega} (D^*(W^{k-1}) -D^*(W^{k}) )\nabla u^{k}\cdot \nabla (u^{k+1}-u^{k}) \di{x}.
\end{align}
 Recalling that the functions $G_i$ are locally Lipschitz (Assumption \ref{A2}), i.e.,
    \begin{equation}
         |G_i(u^{k})-G_i(u^{k-1})|\leq C |u^{k}-u^{k-1}|,\label{xT3e1}
    \end{equation}
    where $C>0$ is independent of $k$ since $|u^k|, |u^{k-1}|\leq m$. Using \eqref{xT3e1}, we get
    \begin{align}
    |D^*(W^{k-1}) -D^*(W^{k})|\leq C|u^k-u^{k-1}|.
\end{align}

Using the aforementioned estimates and \eqref{eq:gradient global bound}, we obtain
 \begin{align}
      \frac{1}{2}\frac{d}{dt} \|u^{k+1}-u^{k}\|_{L^2(\Omega)}^2+\theta \|\nabla u^{k+1}-\nabla u^{k}\|_{L^2(\Omega)}^2 \leq C\int_{\Omega} |u^k-u^{k-1}| |\nabla (u^{k+1}-u^{k})| \di{x},
 \end{align}
 where $C>0$ independent of $k$.
 Using Young's inequality and $u^{k+1}(0,\cdot)=u^k(0,\cdot)$, we get after integrating over time twice for any $t\in(0,T)$: 
\begin{align}
      \|u^{k+1}-u^{k}\|_{L^2(0,t;\Omega)}^2+\theta t \|\nabla(u^{k+1}-u^{k})\|_{L^2((0,t;\Omega)}^2 \leq Ct\|u^k-u^{k-1}\|_{L^2(0,t;\Omega)}^2.\label{eq:timeder of uk+1-uk}
 \end{align}
 Choosing $t^*=\min\{\nicefrac{2}{C},T\}$ yields
 \begin{equation}\label{eq:cauchy estimate}
     \|u^{k+1}-u^{k}\|_{L^2(0,t^*;L^2(\Omega))}\leq \frac{1}{2} \|u^{k}-u^{k-1}\|_{L^2(0,t^*;L^2(\Omega))}.
 \end{equation}
 Thus, we conclude $u^{k}$ is Cauchy in $L^2(0,t^*;L^2(\Omega))$.
 By discretizing $(0,T)$ in small intervals $0<t^*<2t^*<\cdots<T$, in a similar argument that leads to \eqref{eq:cauchy estimate}, we get $u^{k}$ is Cauchy in $L^2(t^*,2t^*;L^2(\Omega))$. Since we have $u^k\in C((0,T)\times \Omega)$,
 and using continuity argument we extend the result \eqref{eq:cauchy estimate} in $(0,T)\times \Omega$, i.e., we get $u^{k}$ is Cauchy in $L^2(0,T;L^2(\Omega))$ (for details see \cite[Remark in Section 2]{lyons2023phase} and \cite[Section 7]{amann2011ordinary} ). 
 \end{proof}

\begin{theorem}[Solvability of $P(\Omega)$]\label{theorem_existance_of_p(omega)}Assume \ref{A1}--\ref{A4} hold.
Then, there exists a $u\in \mathcal{U}$ and a $W\in L^\infty(\timeint\times\Omega;\mathcal{W}^2)$ such that,
     \begin{subequations}   
     \begin{alignat}{4}
         u^{k}&\rightarrow u \hspace{1cm} &&\mbox{strongly in} \hspace{.5cm} &&L^2(\timeint\times\Omega)\label{T3e1},\\
         D^*(W^k) &\rightarrow D^*( W )\hspace{1cm} &&\mbox{strongly in} \hspace{.5cm} &&L^2(\timeint\times\Omega)\label{T3e4},\\
         \nabla u^{k}&\rightharpoonup\nabla u \hspace{1cm} &&\mbox{weakly in} \hspace{.5cm} &&L^2(\timeint\times\Omega)\label{T3e2},\\
          \partial_t u^{k}&\rightharpoonup\partial_t u \hspace{1cm} &&\mbox{weakly in} \hspace{.5cm} &&L^2(\timeint;H^{-1}(\Omega))\label{T3e3}.
     \end{alignat}
     Moreover, $(u,W)$ is a weak solution to the nonlinear parabolic-elliptic system \eqref{homeq1}--\eqref{homeqf} in the sense of \cref{D1} and $u\in L^\infty(\timeint\times\Omega)$.
      \end{subequations}
\end{theorem}
\begin{proof}
    From Lemma \ref{lemma:cauchy}, we get \eqref{T3e1} for some $u\in L^2(\timeint\times\Omega)$.
    Using the weak compactness property (see  \cite[Appendix D, Theorem 3]{evans2010partial}) with the energy estimates \eqref{T2e2} and \eqref{T2e3}, we get both \eqref{T3e2} and \eqref{T3e3} along a subsequence.
    Since $u^k$ is Cauchy, these convergences hold for the full sequences.
    From \eqref{T3e1}, it also follows that $u^k\rightarrow u$ pointwise almost everywhere in $\timeint\times\Omega$.
    Hence, we obtain $u\in L^\infty(\timeint\times\Omega)$ and, more precisely, $|u|\leq m$.    
    Via \cref{L1}, we get the existence of the weak solution $w_i(p_i,\cdot)\in \mathcal{W}$ for \eqref{cell3}--\eqref{homeqf} where $p_i=G_i(u(t,x))$, $i\in \{1,2\}$, for almost all $(t,x)\in\timeint\times\Omega$.
    This then extends to a function $W=(w_1,w_1)\in L^\infty(\timeint\times\Omega;\mathcal{W}^2)$ via \cref{L1} (ii).
    
    Recalling \eqref{xT3e1}, we can use the estimates \eqref{p1} along with \eqref{T3e1}  to show
    \begin{align}
        \lim_{k\rightarrow \infty}\int_0^T\int_\Omega |[D^*(W^k)-D^*(W)]_{i,j}|^2\di{x}&=  \lim_{k\rightarrow \infty}\int_0^T\int_\Omega |[\overline{D}(G_i(u^{k})))-\overline{D}(G_i(u))]_{i,j}|^2\di{x}\nonumber\\
        &\leq C  \lim_{k\rightarrow \infty}\int_0^T\int_\Omega |G_i(u^{k})-G_i(u)|^2 \di{x}\nonumber\\
        &\leq C  \lim_{k\rightarrow \infty}\int_0^T\int_\Omega |u^{k}-u|^2 \di{x}\nonumber\\
        &=0.\label{T3e5}
    \end{align}
 Using \eqref{T3e2} and \eqref{T3e4} together with \cref{lemma:productofsequence}, for any $\phi\in H_0^1(\Omega)$ we have
    \begin{equation}
        \lim_{k\rightarrow \infty}\int_{\Omega}D^*(W^k)\nabla u^{k+1}\cdot\nabla \phi \di{x}= \int_{\Omega}D^*(W)\nabla u\cdot\nabla \phi \di{x}.\label{T3e6}
    \end{equation}
Combining  \eqref{T3e3} and \eqref{T3e6} and then passing $k\rightarrow\infty$ in \eqref{wfs1}, we see that $u$ satisfies \eqref{wf}.
Hence, we have that $(u,W)$  solves Problem $P(\Omega) $ in the sense of Definition \ref{D1}.
\end{proof}
\begin{lemma}[Uniqueness]\label{lemma:uniq}
    Assume \ref{A1}--\ref{A4} hold. Then the weak solution in the sense of \cref{D1} is unique.
\end{lemma}
\begin{proof}
    From \eqref{eq:c2holderbound} and \eqref{eq:gradient global bound} , we have $\nabla u^k\to \nabla u$ strongly in $L^2(\timeint\times \Omega)$ and $\nabla u\in{L^{\infty}(\timeint\times\Omega)}.$
  Then, uniqueness follows via standard arguments.
  Indeed, if $u_1,u_2$ solve Problem $P(\Omega)$, then by using the Lipschitz continuity of $D^*(W)$ (see Lemma \eqref{L2}) and standard energy estimates for $u_1-u_2$, we can conclude that $u_1=u_2$ for almost everywhere in $\timeint\times\Omega$.
  \end{proof}
\section{Numerical simulation}\label{section_simulation}
In this section, we demonstrate that the iteration scheme \eqref{Eq:IterationScheme} can be used as a numerical method to approximate the solution of Problem $P(\Omega)$. 
Additionally, we use this numerical method to illustrate the impact of different microscopic geometries and to investigate how the effective dispersion tensor, $D^*$, governs the macroscopic behavior of the system. 
To accomplish this, we make use of finite element method solvers in FEniCS \cite{logg2012automated,Langtangen2016} with Lagrange polynomials of degree one as the basis elements for both the elliptic and parabolic problems. 
We begin by discretizing the macroscopic and microscopic domains.
By first fixing an initial iteration $u^0$, we can then solve the linear elliptic cell problems given by \eqref{kc1}--\eqref{kc4} in the microscopic domain for each point in the macroscopic domain. 
Please note that this step can be perfectly parallelized as each cell problem is independent. 
By calculating $D^*(W^k)$ given by \eqref{kmacrodiff}, we can then solve the linear parabolic problem \eqref{kp1}--\eqref{kp3}, where we make use of an implicit Euler discretization to handle the time dependence. 
Then, using this numerical solution, we repeat these calculations for the next iteration. 
More concisely, given the numerical solution $u^k$, we calculate $u^{k+1}$ by solving the linear system \eqref{Eq:IterationScheme}. 
We continue this iteration method until a desired error tolerance between successive iterations (measured in the $L^2(0, T; L^2(\Omega))$ norm) is reached.
We outline this approach in Algorithm \ref{alg:cap}. 

To proceed with our investigation, we define two specific geometries (henceforth named Geometry 1 and 2) which have different and interesting effects on the effective dispersion tensor.
 For Geometry 1, we choose  $Y = (0,1)^2 \setminus  \mathcal{B}_{0.25}((0.5, 0.5))$ where  $\mathcal{B}_{r}((y_1, y_2))$ denotes the closed disk with radius $r$ and center $(y_1, y_2)$ (for a visualization, see the top row of Figure \ref{stokegeometry1}). 
 For Geometry 2, we let $Y = (0,1)^2 \setminus  (\mathcal{R}_{1} \cup \mathcal{R}_{2})$ with rectangles $\mathcal{R}_{1}:= [0.1, 0.9] \times [0.1, 0.2]$ and $\mathcal{R}_{2}:= [0.1, 0.9] \times [0.8, 0.9]$(for a visualization, see the bottom row of Figure \ref{stokegeometry1}). 
 
When choosing model ingredients, one difficulty is ensuring that the microscopic drift, $B(y)$, is both interesting, physically relevant, and satisfies the assumptions {\ref{A3}} for each geometry.
To accomplish this, we take the velocity field as the solution to the following Stokes problem:

 \begin{mybox}{Stokes problem for the drift velocity}
 \vspace{-.4cm}
 \begin{subequations}\label{Eq:Stokes_System}
 \begin{alignat}{2}
 \label{stoke1}- \mu \Delta B + \nabla p &= F(y) \quad & &\text{in}\;\; Y,\\
 \label{stoke2}\text{div}\, B &= 0 \;\; & &\text{in}\;\; Y,\\
\label{stoke3}  B &= 0 \;\; & &\text{on} \;\;\Gamma_N,\\
 \label{stoke4} y&\mapsto B(y) &\;&\text{is} \; Y\text{-periodic}, 
 \end{alignat}
  \end{subequations}
\end{mybox}
\noindent where $Y$ is either Geometry 1 or Geometry 2.
To ensure stability while solving  \eqref{Eq:Stokes_System}, we use the Taylor-Hood elements in FEniCS, where second-order polynomials are used as basis functions for the velocity field and first-order polynomials are used for the pressure. We refer the reader, for instance, to \cite{arnold1984stable} for more information on the use of stable finite elements for the computation of Stokes flow.

For all simulations to follow, we choose $\mu=0.01$ and
\[F:Y\to\mathbb{R}^2 \mbox{ given by }F(y) = (10\sin(2 \pi y_1) \sin(2\pi y_2), 10\sin(2 \pi y_1) \cos(2\pi y_2))\]
for $y=(y_1,y_2)\in Y$.
The velocity vector $B(y) = (B_{1}(y), B_{2}(y))$ corresponding to Geometry 1 and Geometry 2 is shown in Figure \ref{stokegeometry1}.

 \begin{figure}[h!]
	\centering
    \includegraphics[width=0.4\textwidth]{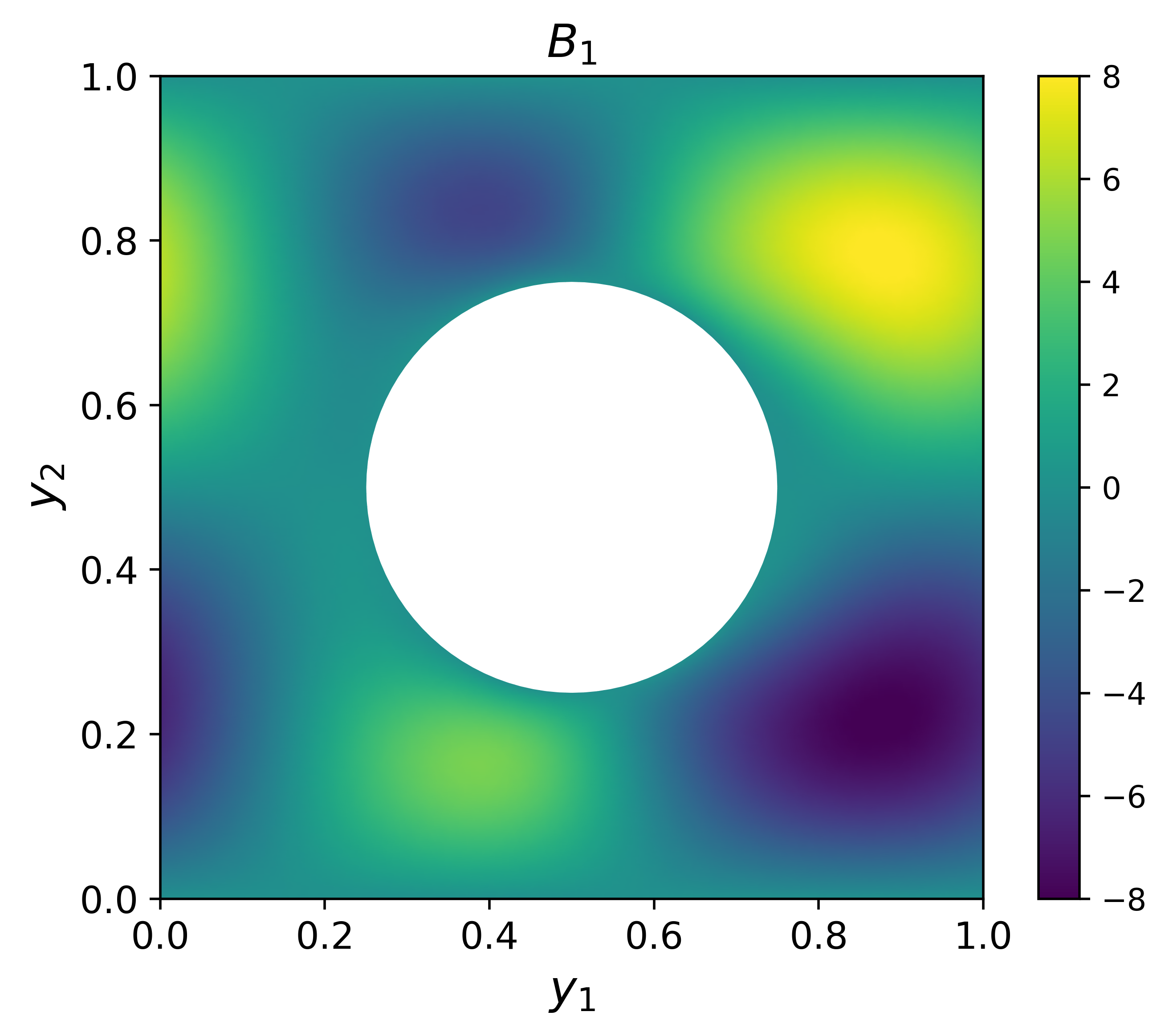}
	\hspace{0.5cm}
    \includegraphics[width=0.4\textwidth]{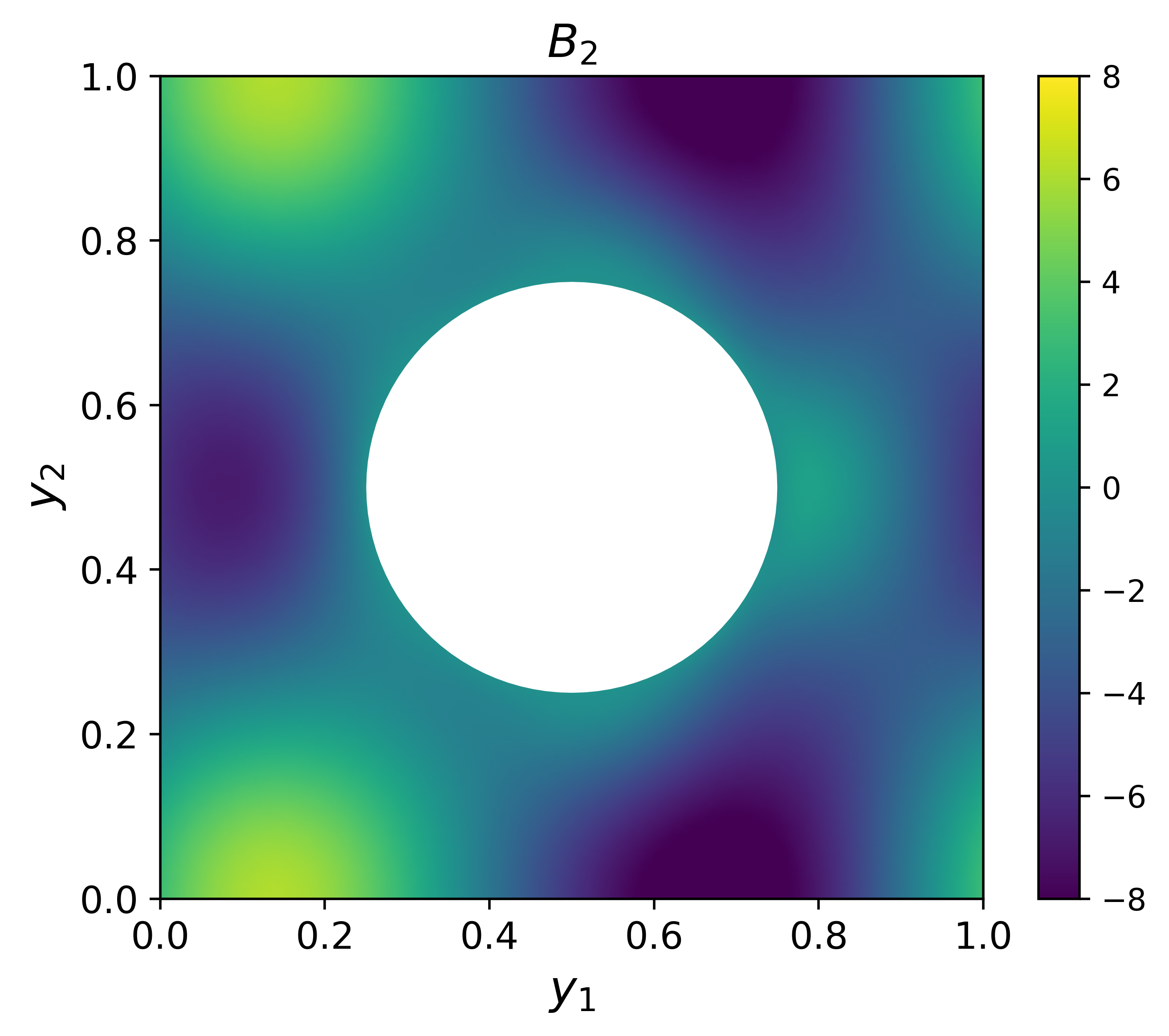}
    \includegraphics[width=0.4\textwidth]{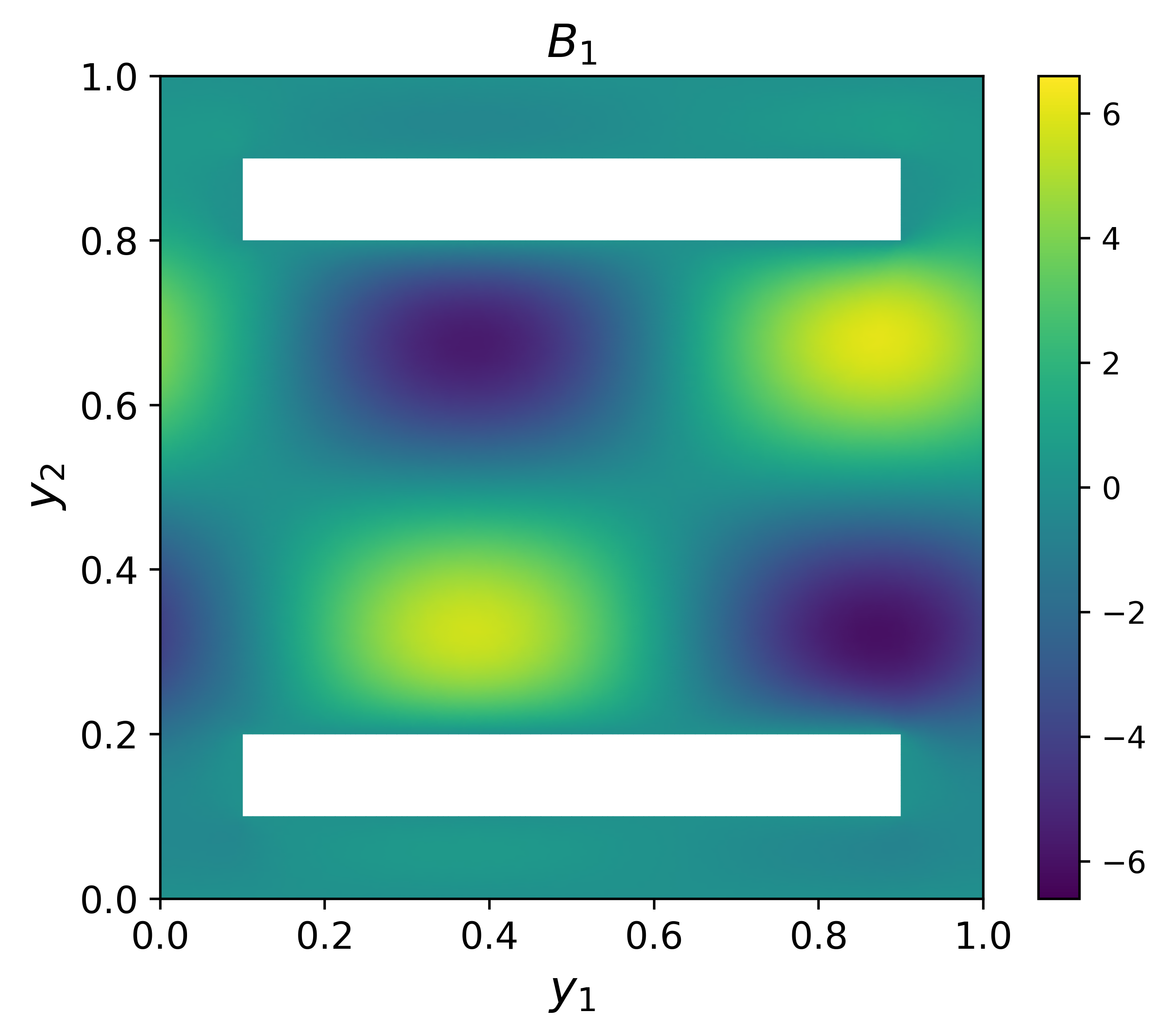}
	\hspace{0.5cm}
    \includegraphics[width=0.4\textwidth]{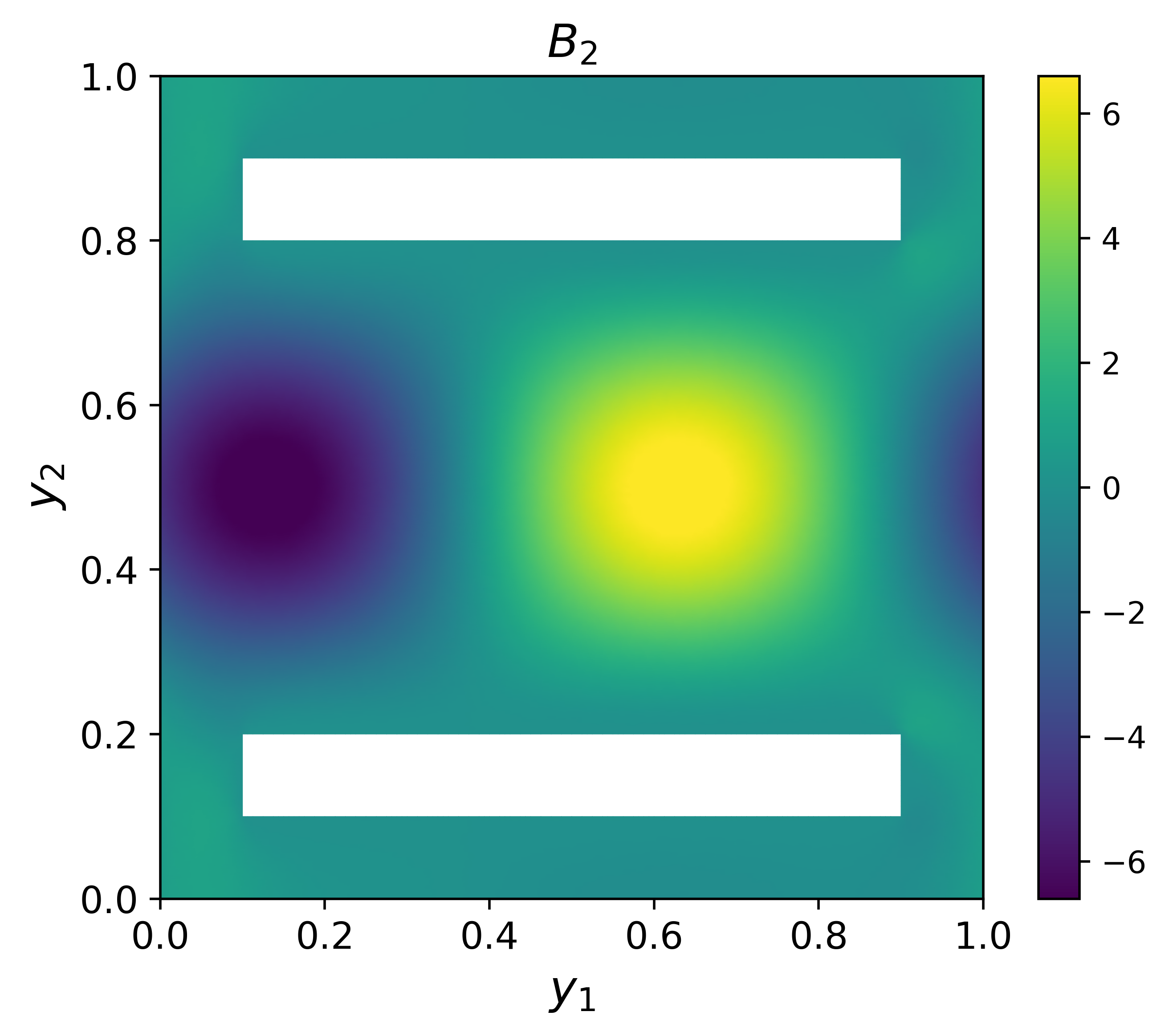}
	\caption{ Top: solution to the Stokes problem, $B_1(y)$ (left) and $B_2(y)$ (right), in Geometry 1. Bottom: solution to the Stokes problem, $B_1(y)$ (left) and $B_2(y)$ (right), in Geometry 2.}
	\label{stokegeometry1}
\end{figure}

 \begin{algorithm}[h!]
\caption{Procedure to compute the weak solution to \eqref{kc1}--\eqref{kp3}.} \label{alg:cap}
\begin{algorithmic}[1]
\State \texttt{Discretize the space micro domain $Y$ and macro domain $\Omega$}
\State \texttt{Discretize the time domain $[0, T]$ with step size $\Delta t$}
\State \texttt{Solve the Stokes problem \eqref{stoke1}-\eqref{stoke4} to get $B(y)$}
\State \texttt{Set initial iteration guess $u^0$}
\State \texttt{Choose data $f, D, G_i, g$}
\State \texttt{Set tolerance value $\epsilon$}
\State \texttt{Set the maximum number of iterations,  Maxiter.}
\State \texttt{Initialize iteration and time. i.e. $iter =0$, $t=0$}
\While{$ iter < Maxiter$}
\State \texttt{Set $u_{\text{old}}=g$}
\State \texttt{Set $u = [u_{\text{old}}]$}
\For{\texttt{each time discrete node  on time domain}}
\For{\texttt{each node on macroscopic grid}}
        \State \texttt{Solve for $(w_1, w_2)$ using $(G_1(u_{\text{old}}), G_2(u_{\text{old}}))$}
      \EndFor
      \State \texttt{Compute  $D^*$ from $(w_1, w_2)$}
       \State \texttt{Solve for $u_{\text{new}}$ using $D^*$}
        \State \texttt{append $u$ with $u_{\text{new}}$}
       \State \texttt{$u_{\text{old}} \leftarrow u_{\text{new}}$}
    \EndFor
      \If{$\|u-u^0\|< \epsilon$}
    \State \texttt{Stop}
\EndIf 
\State{$u^0\leftarrow u$}
\EndWhile
\end{algorithmic}
\end{algorithm}

\subsection{Dispersion tensor}\label{SSec:Influence_on_dispersion}
We now study the behavior of the effective dispersion tensor with respect to the nonlinear drift interactions in both geometries.
To begin with, we divide the interval $[-10, 10]$ into $101$ equidistant nodes denoted $p_i$.
We then solve the auxiliary cell problem given by \eqref{aux1}-\eqref{aux3} for each $p_i$.
  To enforce the fact that the cell solution has zero average, we similarly use the Lagrange multiplier method as in \cite{formaggia2002numerical}.
With the solutions to the cell problems,  we compute the entries of the effective dispersion tensor using Equation \eqref{macrodiff}.
These results are then interpolated and the components of $D^*$ are plotted as a function of $p\in[-10,10]$ in \cref{fastdiffusion,slowdiffusion}.
{The parameter $p$ scales the microscopic drift velocity (see System \eqref{system:cellproblem_with_p}), playing this way the role of a local (microscopic) Peclet-type number for a given velocity field $B$.
Variations in $p$ weaken or strengthen the macroscopic drift effect in a rather non-intuitive way.
Our numerical studies seem to indicate that geometry 1 is likely to lead to stronger dispersion effects at least for the given velocity field.
We show in Figure \ref{fastdiffusion} how such effects influence the entries of the dispersion tensor.
This type of observation is very similar to the discussion presented in Section 5 of \cite{ALLAIRE20102292}.}
\paragraph{Case 1: Fast diffusion}
We first consider the case where the microscopic diffusion is comparable in magnitude to the velocity field $B(y)$.
To this end, we choose the diffusion matrix
 \begin{align*}
 D(y) =
 \begin{pmatrix}
2 + \sin(\pi y_1)\sin(\pi y_2) & 0\\
0 & 2+ \sin(\pi y_1)
\end{pmatrix},\quad y=(y_1,y_2)\in Y,
 \end{align*}
which satisfies $1\leq \det D(y)\leq 9$ everywhere in $Y$.
\begin{figure}[h!]
	\centering		
    \includegraphics[width=0.45\textwidth]{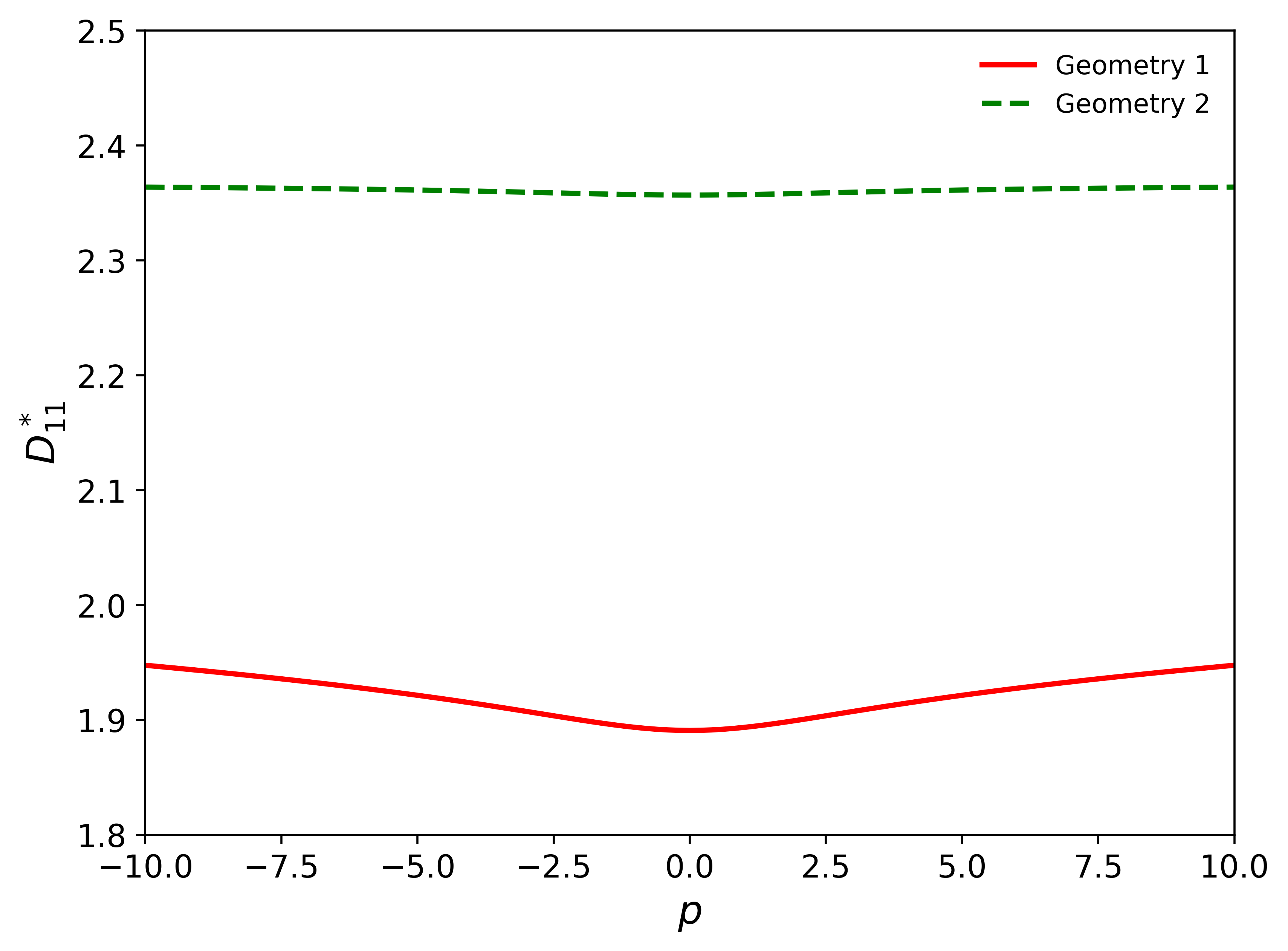}
	\hspace{0.3cm}
    \includegraphics[width=0.45\textwidth]{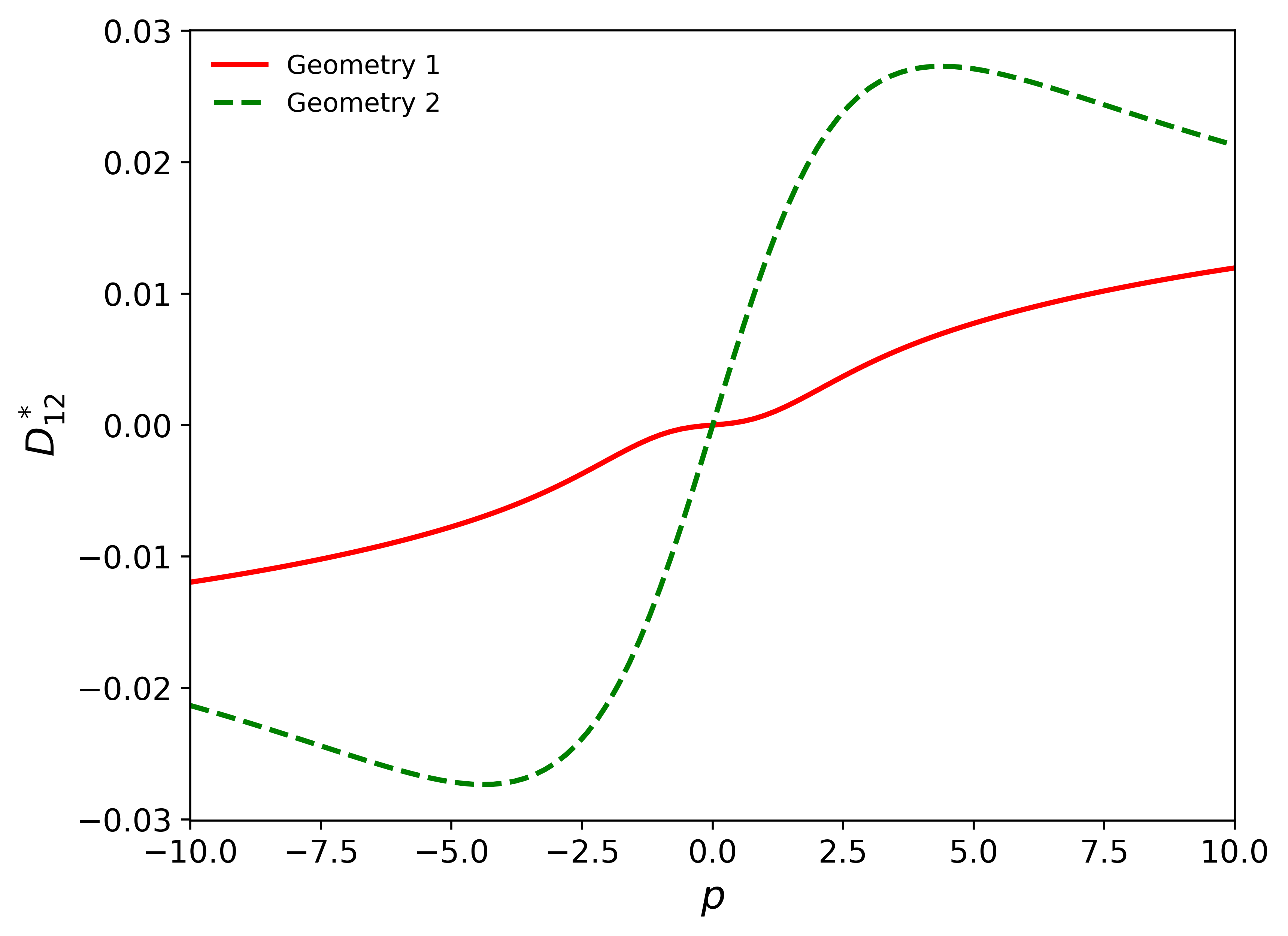}\\
    \includegraphics[width=0.45\textwidth]{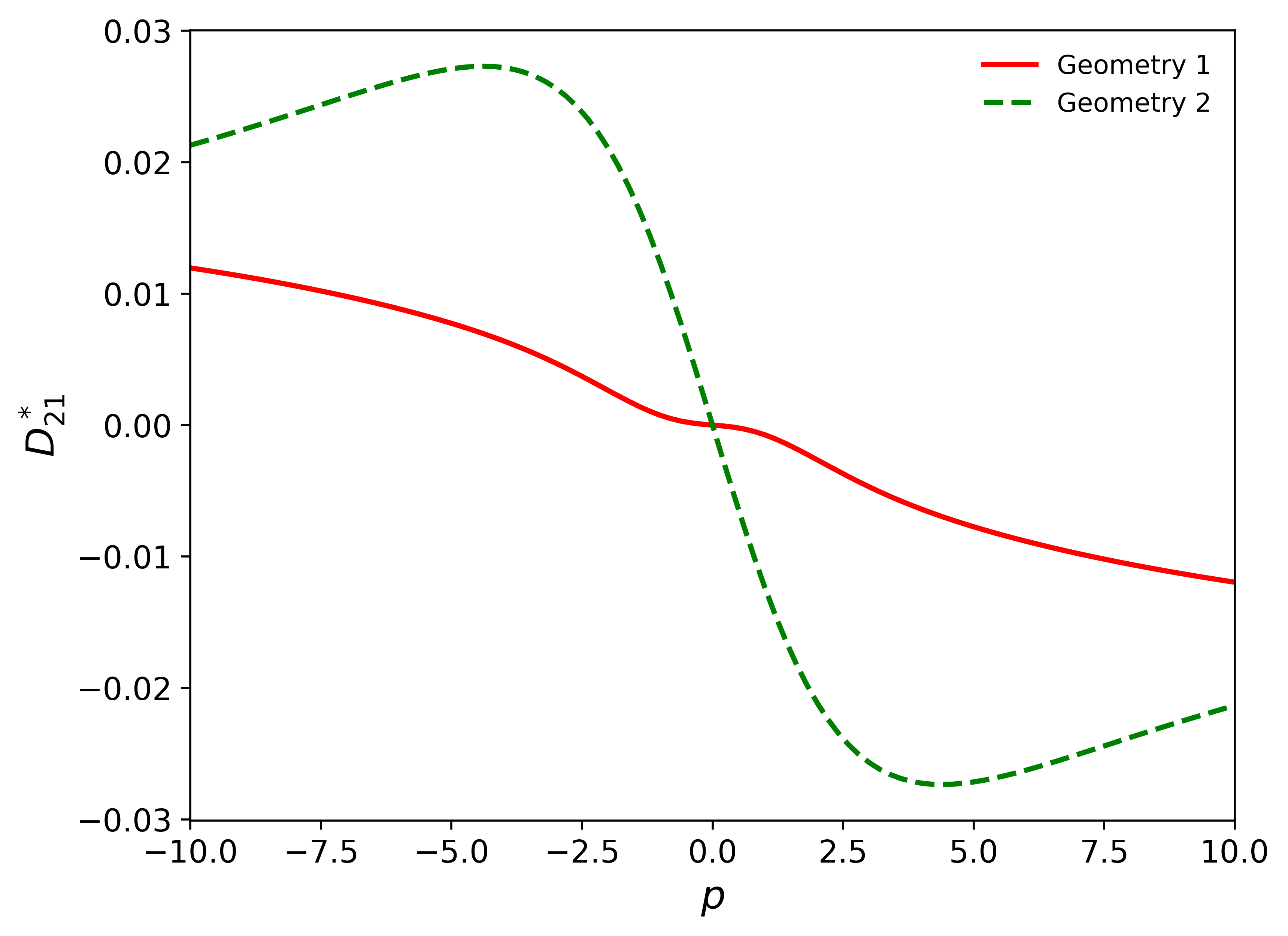}
	\hspace{0.3cm}
    \includegraphics[width=0.45\textwidth]{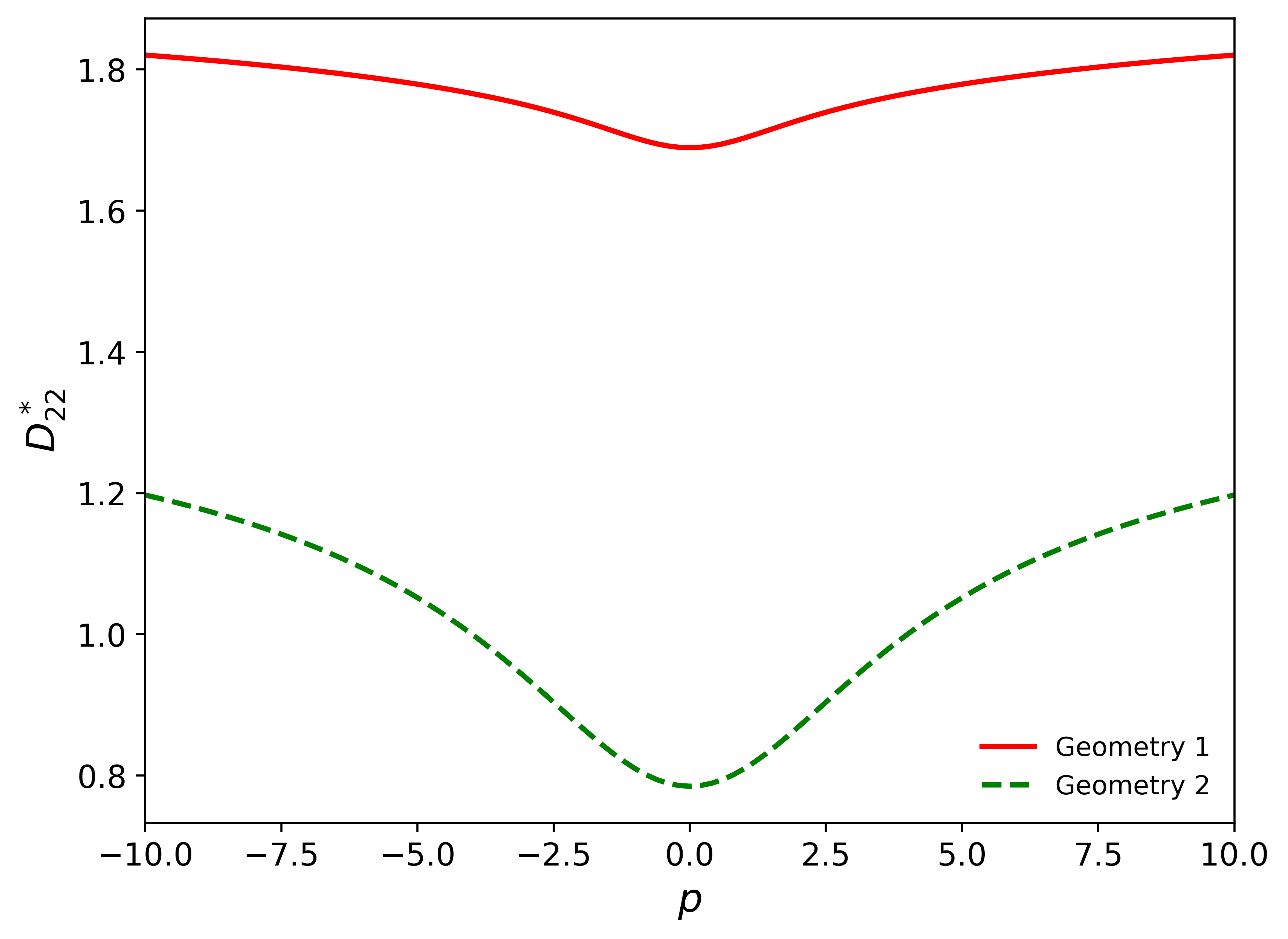}
	\caption{Comparison of the components of $D^*(W)$ for the fast diffusion case with both Geometry 1 (solid red line) and Geometry 2 (dashed green line). Here, we plot the components $D^*_{11}$ (top left), $D^*_{12}$ (top right),  $D^*_{21}$ (bottom left), and $D^*_{22}$ (bottom right) as functions of {the parameter $p$ which represents a type of local (microscopic) Peclet number.}}
	\label{fastdiffusion}
\end{figure}
Looking at \cref{fastdiffusion}, we notice that the main diagonal entries $D^*_{11}$ and $D^*_{22}$ are symmetric with respect to the vertical axis while the off-diagonal entries $D^*_{12}$ and $D^*_{21}$ are symmetric about the origin.
However, we note that the off-diagonal entries are close to $0$ and have therefore only a small impact on the macroscopic dispersion.
This is a consequence of both geometries being symmetric: in non-symmetric setups, e.g., angled rectangles or ellipses, the off-diagonal entries would play a bigger role. 
The particular case $p=0$ corresponds to ``no drift" and, in both geometries, it is the value with the slowest dispersion.

We can also see that while $D^*_{11}$ and $D^*_{22}$ are similar (both quantitatively and qualitatively) in the first geometry, they differ quite a bit in the second geometry.
Focusing on Geometry 2, the dispersion in the $y_1$ direction appears to be much faster than in the $y_2$ direction as can be seen by $D^*_{11}$ being almost twice as large as $D^*_{22}$.
This is to be expected since the geometric setup with the two rectangles impedes vertical flows much more than lateral flows.
More interestingly, $D^*_{11}$ is close to constant in $p$ while, on the other hand, $D^*_{22}$ is much more dynamic with changes of up to $50\%$ relative to the lowest value for $p=0$.
This shows that the drift interaction can play a large role in facilitating vertical flows for Geometry 2.

\paragraph{Case 2: Slow diffusion}
Now, we consider the slow diffusion case where the microscopic diffusion is small relative to the velocity field.
We compute the components of the effective dispersion tensor and compare the results in both geometries in \cref{slowdiffusion}.
For this case, we choose the diffusion tensor to be 
 \begin{align*}
 D(y):=
 \begin{pmatrix}
0.05 + \frac{1}{50}\sin(\pi y_1)\sin(\pi y_2) & 0\\
0 & 0.05 + \frac{1}{50}\sin(\pi y_1)
\end{pmatrix},\quad y=(y_1,y_2)\in Y,
\end{align*}
which satisfies $0.009\leq\det D(y)\leq0.0049$ everywhere in $Y$.
\begin{figure}
	\centering
    \begin{subfigure}[b]{0.48\textwidth}
        \centering
        \includegraphics[width=\textwidth]{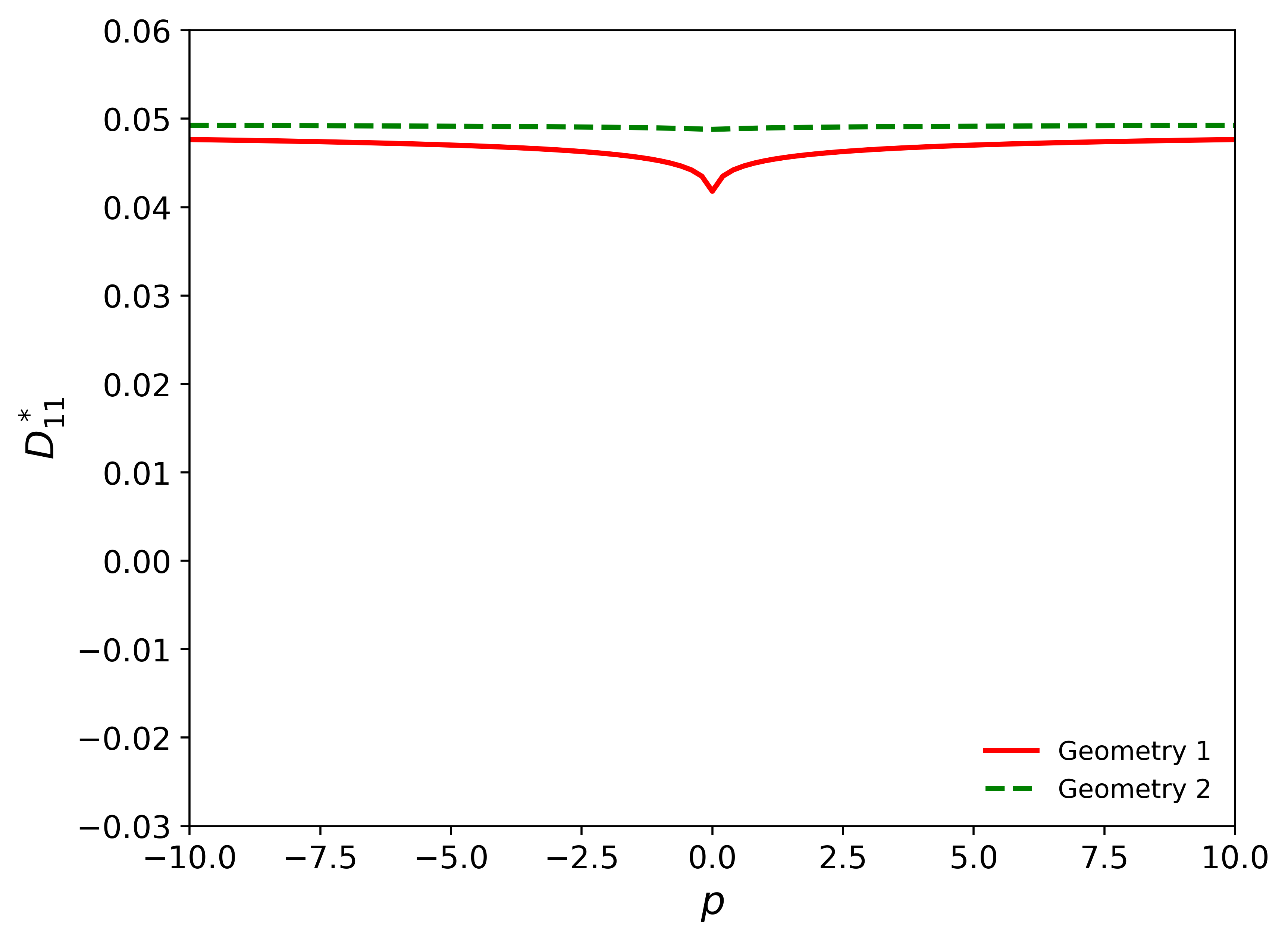}
    \end{subfigure}
    \hfill
    \begin{subfigure}[b]{0.48\textwidth}
        \centering
        \includegraphics[width=\textwidth]{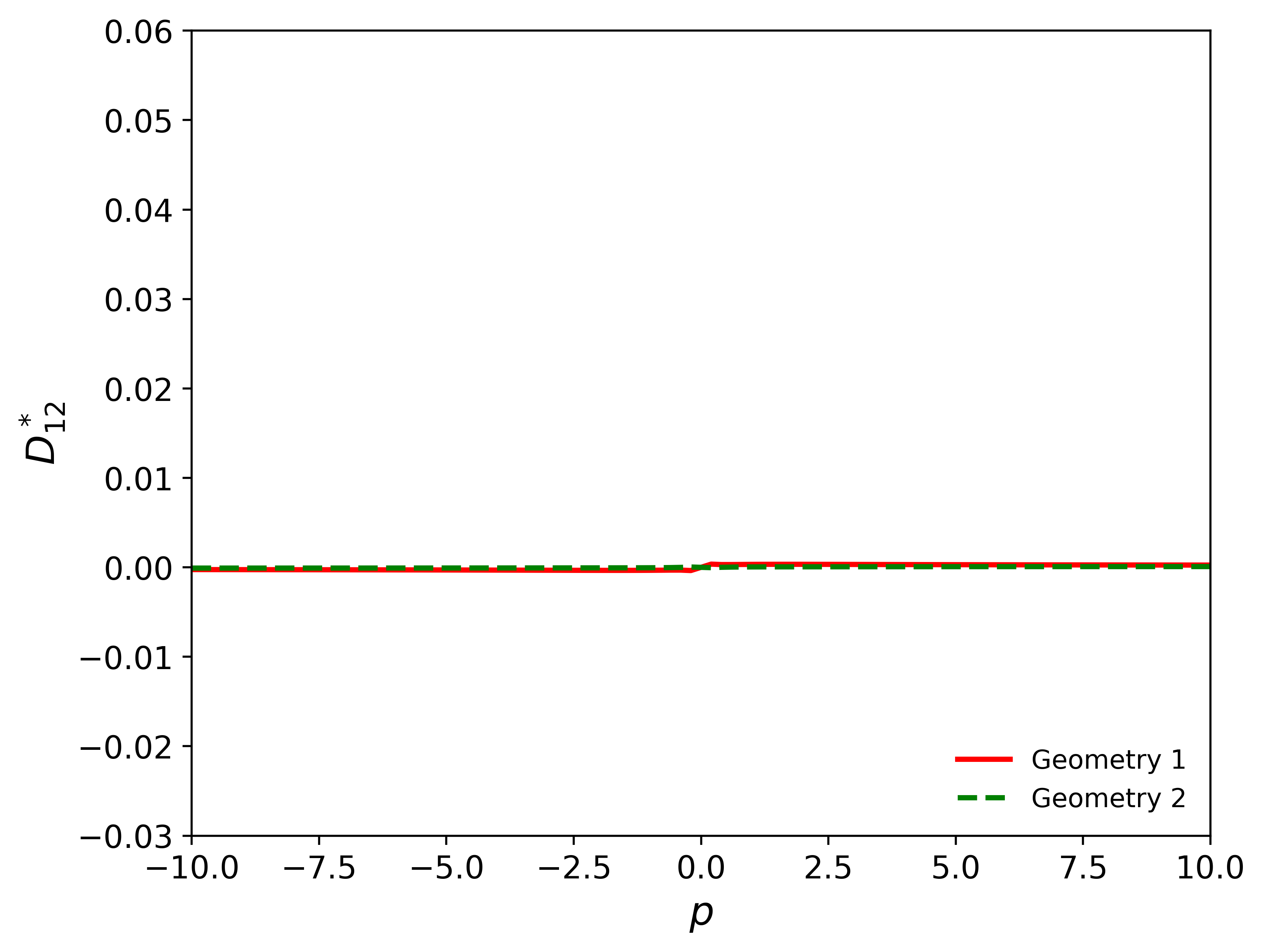}
    \end{subfigure}\\
    \begin{subfigure}[b]{0.48\textwidth}
        \centering
        \includegraphics[width=\textwidth]{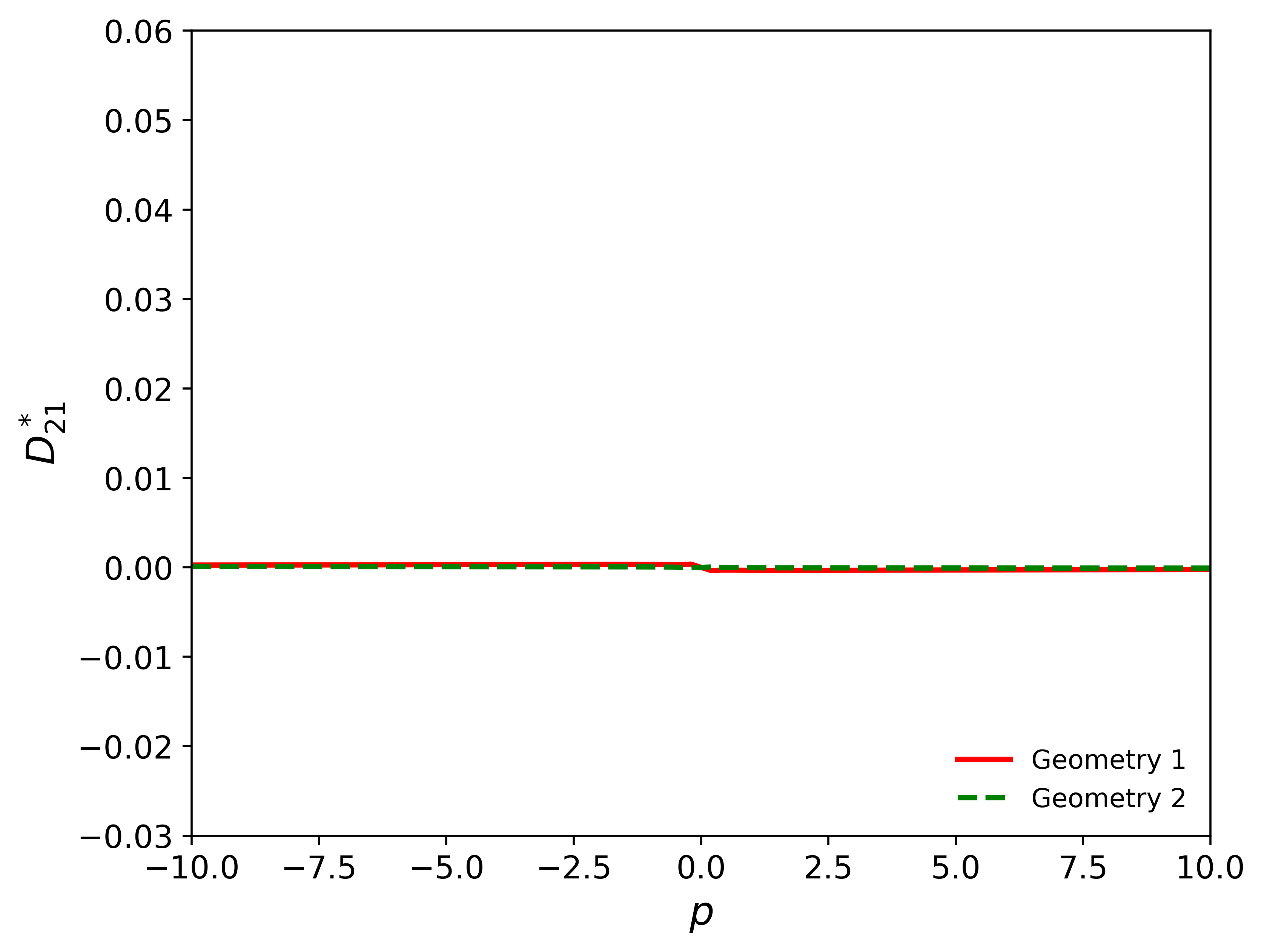}
    \end{subfigure}
    \hfill
    \begin{subfigure}[b]{0.48\textwidth}
        \centering
        \includegraphics[width=\textwidth]{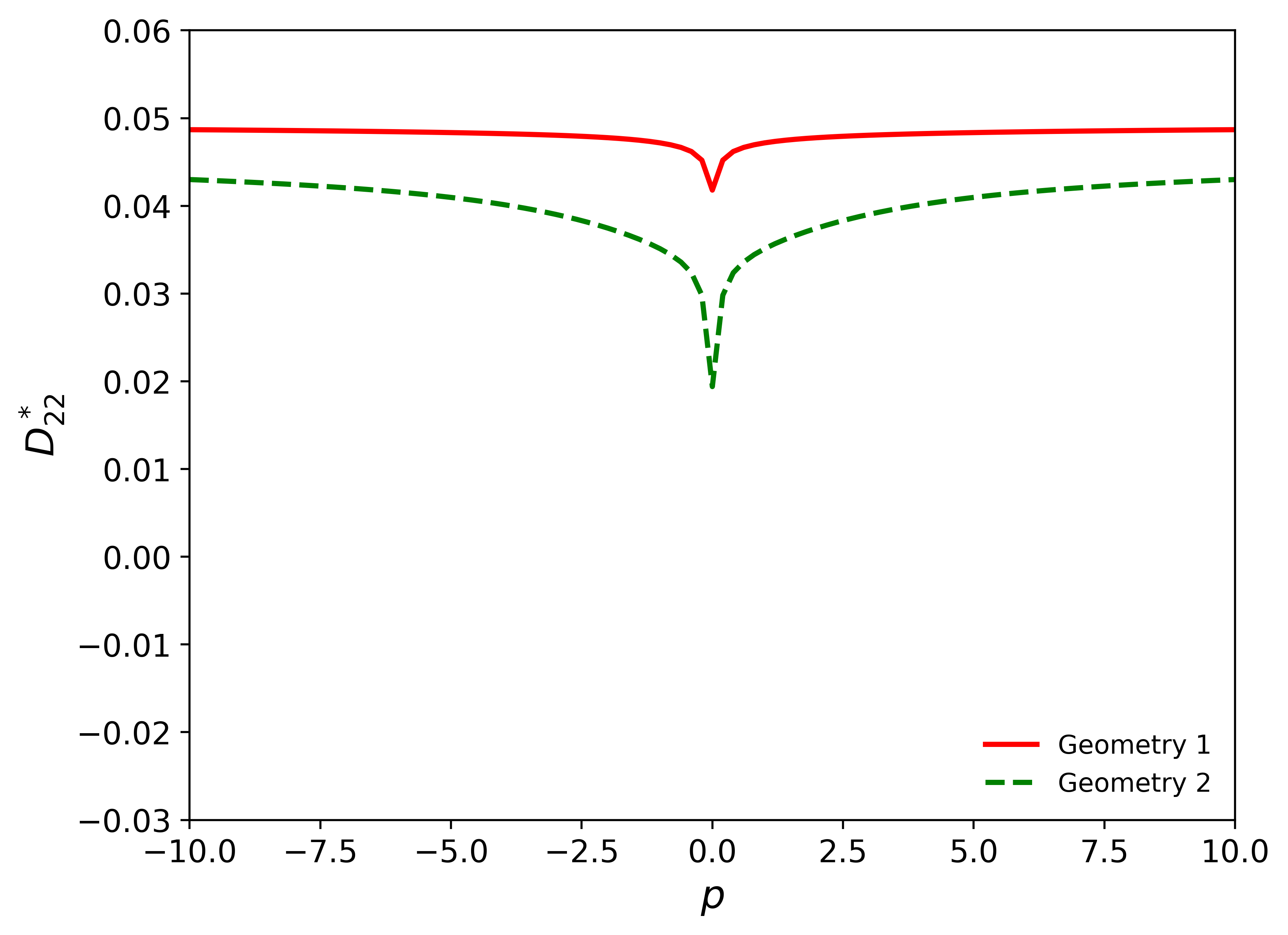}
    \end{subfigure}
		\caption{Comparison of the components of $D^*(W)$ for the slow diffusion case with both Geometry 1 (solid red line) and Geometry 2 (dashed green line). Here, we plot the components $D^*_{11}$ (top left), $D^*_{12}$ (top right),  $D^*_{21}$ (bottom left), and $D^*_{22}$ (bottom right) as functions of the parameter $p$.}
		\label{slowdiffusion}
\end{figure}
Comparing \cref{slowdiffusion} with \cref{fastdiffusion}, we can observe very similar trends qualitatively although the values are, of course, much lower due to the slow diffusion. 
At first glance, it appears that there are cusps forming at the critical point $p=0$ in $D^*_{11}$ for Geometry 1 and $D^*_{22}$ for both geometries.
However, a closer look shows that the transition is in fact smooth, see \cref{slowdiffusionzoom}.
These observed ``cusps" do however point to the fact that the slow diffusion case is much more volatile with respect to small changes close to the zero drift case $p=0$ when compared to the fast diffusion case.
Also, the relative changes in the dispersion values are higher than in the case of fast diffusion, e.g., in Geometry 2 the value $D^*_{22}$ experiences changes of more than 100$\%$ relative to the lowest value for $p=0$. 

\begin{figure}[ht]
	\centering
    \includegraphics[width=0.4\textwidth]{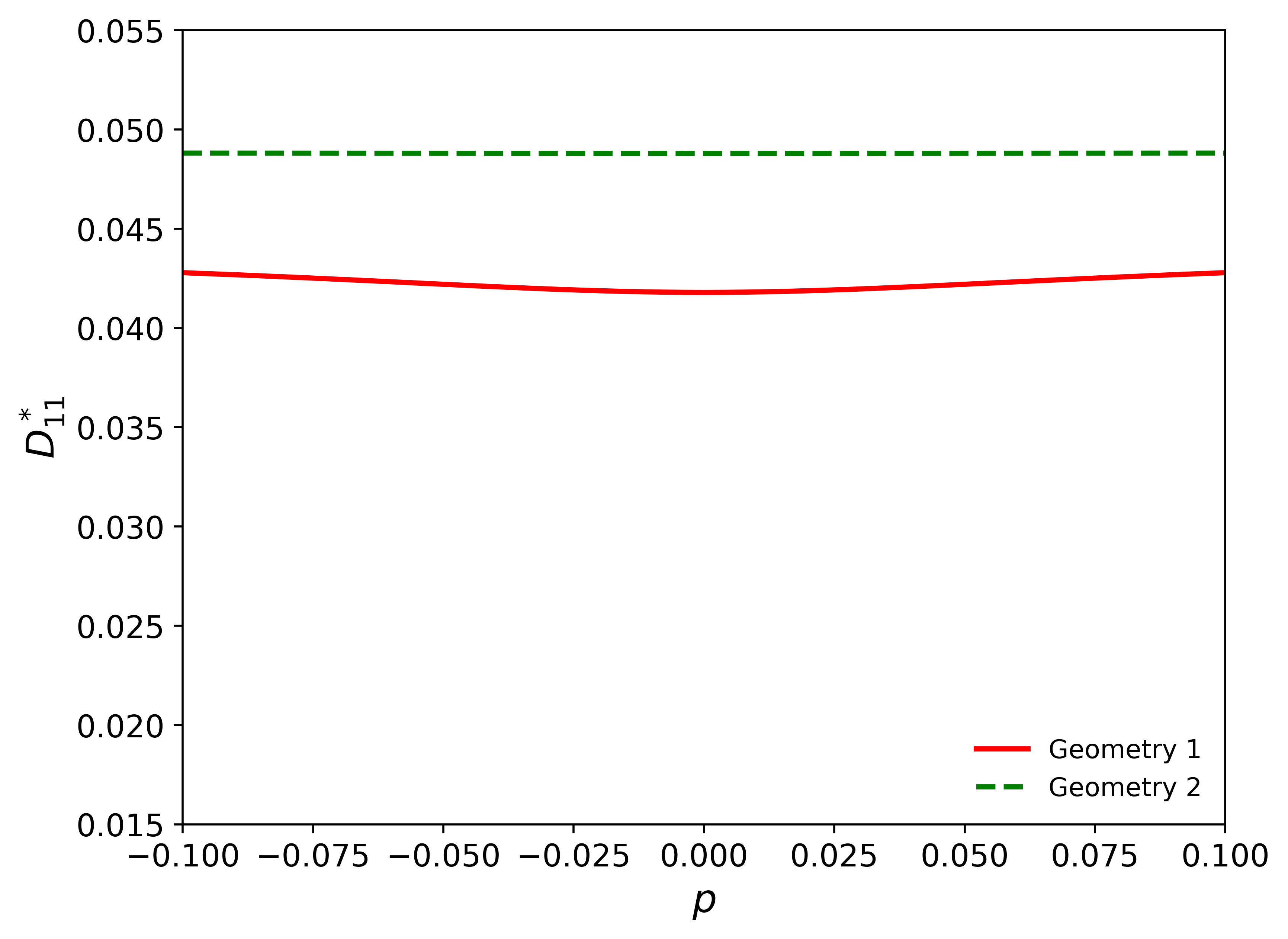}
	\hspace{1cm}
    \includegraphics[width=0.4\textwidth]{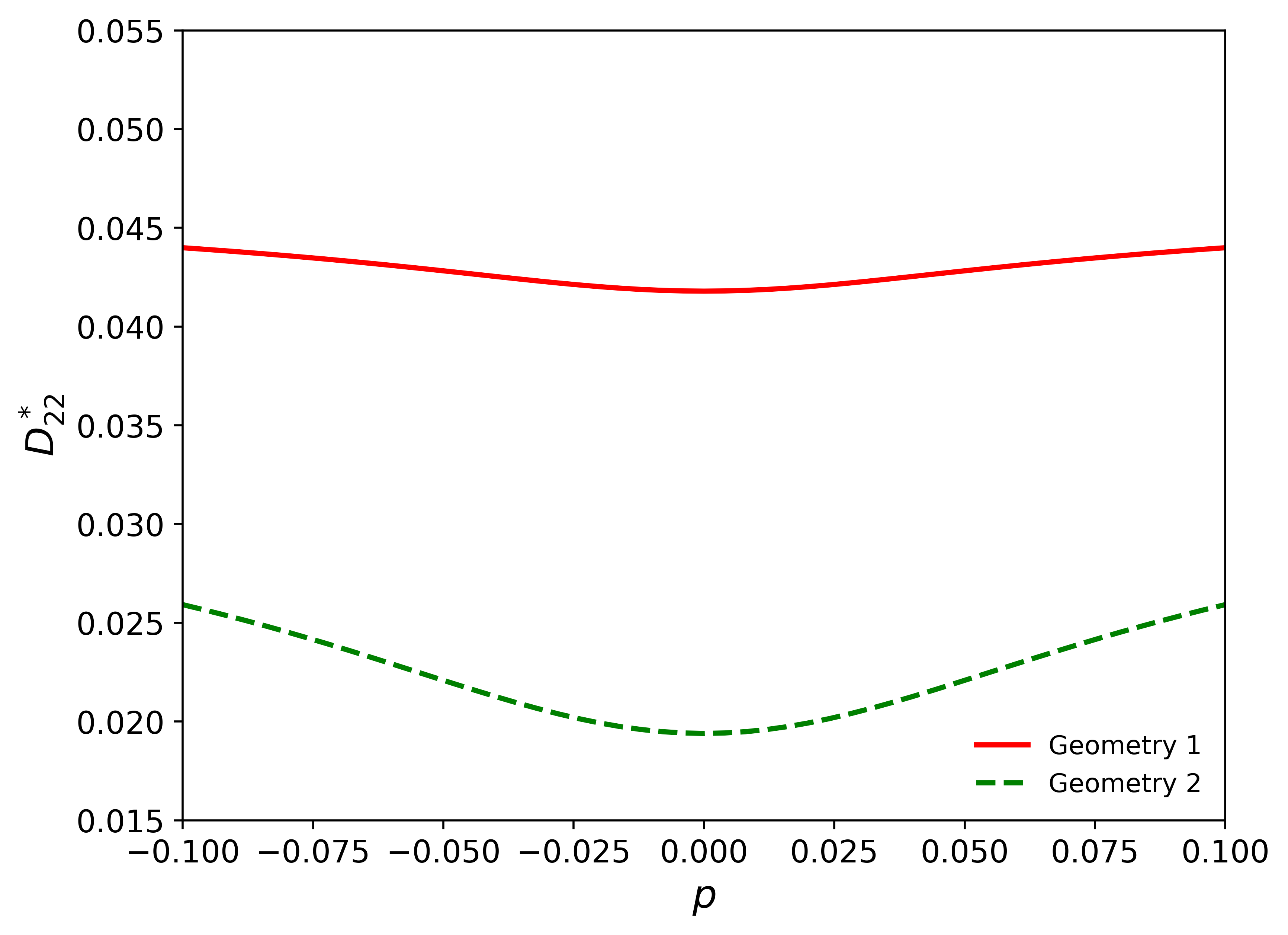}
	\caption{Comparison of two components of $D^*(W)$,  $D^*_{11}$ (left) and $D^*_{22}$ (right), in the neighbourhood of $p=0$ for the slow diffusion case.}
	\label{slowdiffusionzoom}
\end{figure}

\subsection{Macroscopic solution}
The goal of this section is to illustrate how the behavior of the dispersion tensor presented above affects the solution to the macroscopic equation.
In this section, we fix $\Omega = (0,1) \times (0,2)$ and the initial profile of concentration  
\begin{align}g(x_1, x_2) =  \begin{cases} \exp{(-10((x_1-0.5)^2 + (x_2-0.5)^2))}, \;\;\text{if} \;\; (x_1, x_2) \in  \mathcal{B}_{0.25}((0.5, 0.5)),\\
0, \;\;\text{otherwise}.
\end{cases} \end{align} 
We also fix a constant source term 
$f:\Omega\to\mathbb{R}$ given by \begin{align} f(x_1, x_2) = \begin{cases} 1000, \;\;\; \text{if} \;\;\; (x_1, x_2)\in \mathcal{B}_{0.25}((0.5, 0.5)),\\
0,  \;\;\;\text{otherwise}.
\end{cases}\end{align}
The iteration scheme requires an initial guess $u^0$ for which we choose $u^0(t, x) = g$.
For the microscopic ingredients, we take the same diffusion matrix as in the fast diffusion case above for all simulations. 
The velocity field, $B(y)$, is treated in the same manner as before. 

As explained previously, we iterate the FEniCS solvers until the maximum number of iterations is achieved or the error $e_k:=||u^{k+1}-u^k||_{L^2(0, T,  L^2(\Omega))}
< \mbox{tol}$, where $\mbox{tol}$ is a small tolerance value prescribed {\em a priori} and $k\in \mathbb{N}$ is the iteration index. 
For all simulations, we choose ${\rm tol}=10^{-7}$ and use $50$ space nodes in both directions in the macroscopic domain $\Omega$.
Additionally, we take $G_1(u) = G_2(u) = G(u)$ for these simulations. 

\paragraph{Effect of the microscopic geometries}
In Figure \ref{macrosolution}, heat maps of the concentration profile $u$ at time $T=2$ with $G(u) = 1-2u$ using both microscopic geometries are displayed.
Comparing the first and second plots in Figure \ref{macrosolution}, similar dispersive behavior of the concentration profiles is exhibited in both geometries. 
However, the concentration profile corresponding to Geometry 1 disperses faster than that corresponding to Geometry 2. 
This is expected as, due to the long horizontal rectangular obstacles in Geometry 2, the flow is hindered in the vertical direction. 
The third plot in  Figure \ref{macrosolution} shows the difference between both solutions corresponding to the two different microscopic geometries. 
Looking at this difference, we observe that the concentration in the neighborhood of the point $(0.5, 0.5)$ is higher for Geometry 2 than for Geometry 1.     
To quantify this observed behavior, we calculate the total amount of the mass $M(t)$ of concentration $u$ over time on the upper subdomain, $(0,1)\times (1,2)$, using the following formula: 
\begin{align}
M(t) := \int_1^2  \int_0^1 u(t, x_1, x_2)\di{x_1}\di{x_2}. 
\end{align}
We then plot the indicator $M(t)$ over time $t$ for both choices of geometry in the left plot of Figure \ref{mass_vertical_domain}.

\begin{figure}[ht]
		\centering	
    \includegraphics[width=0.30\textwidth]{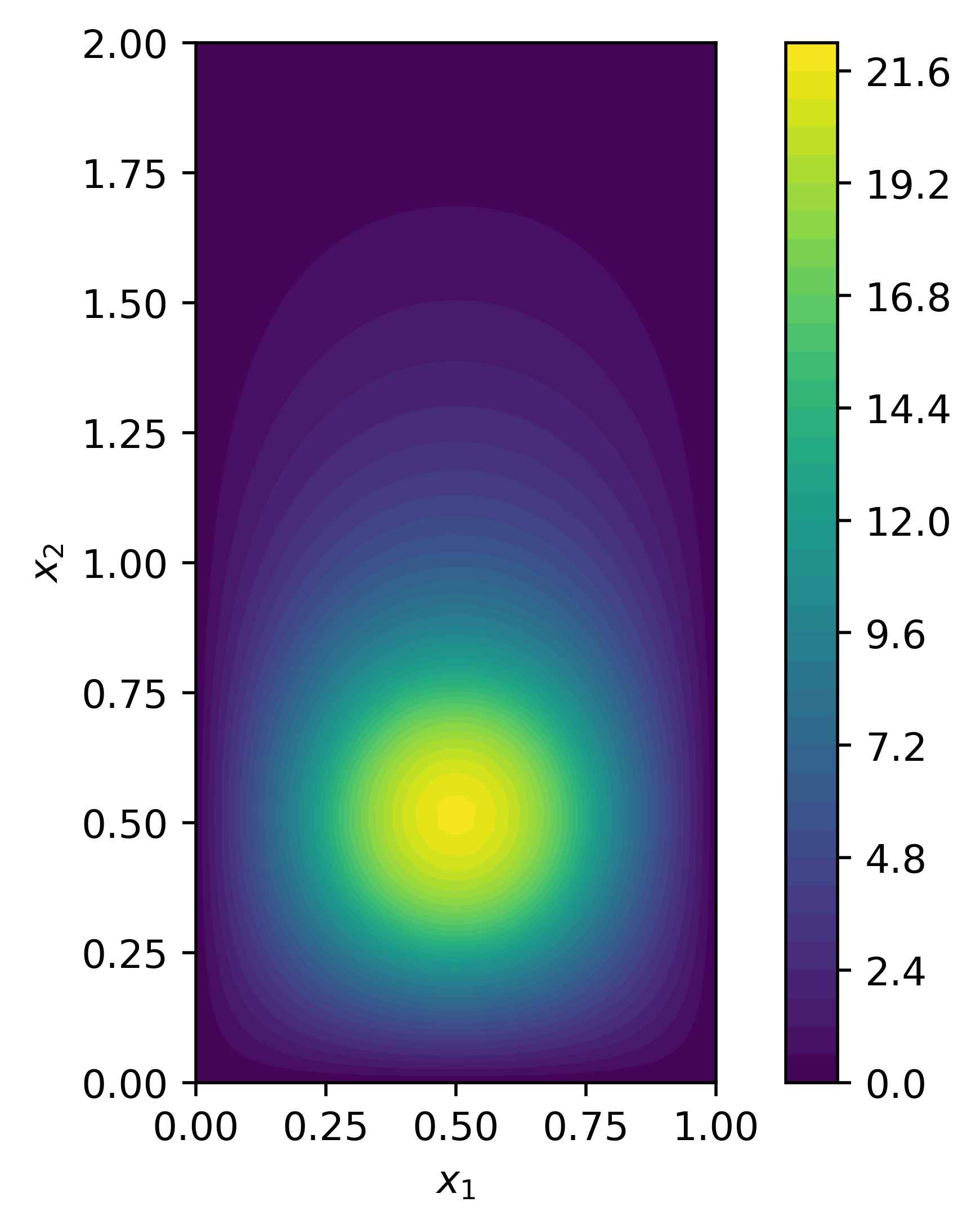}
		\hspace{0.001cm}
       \includegraphics[width=0.30\textwidth]{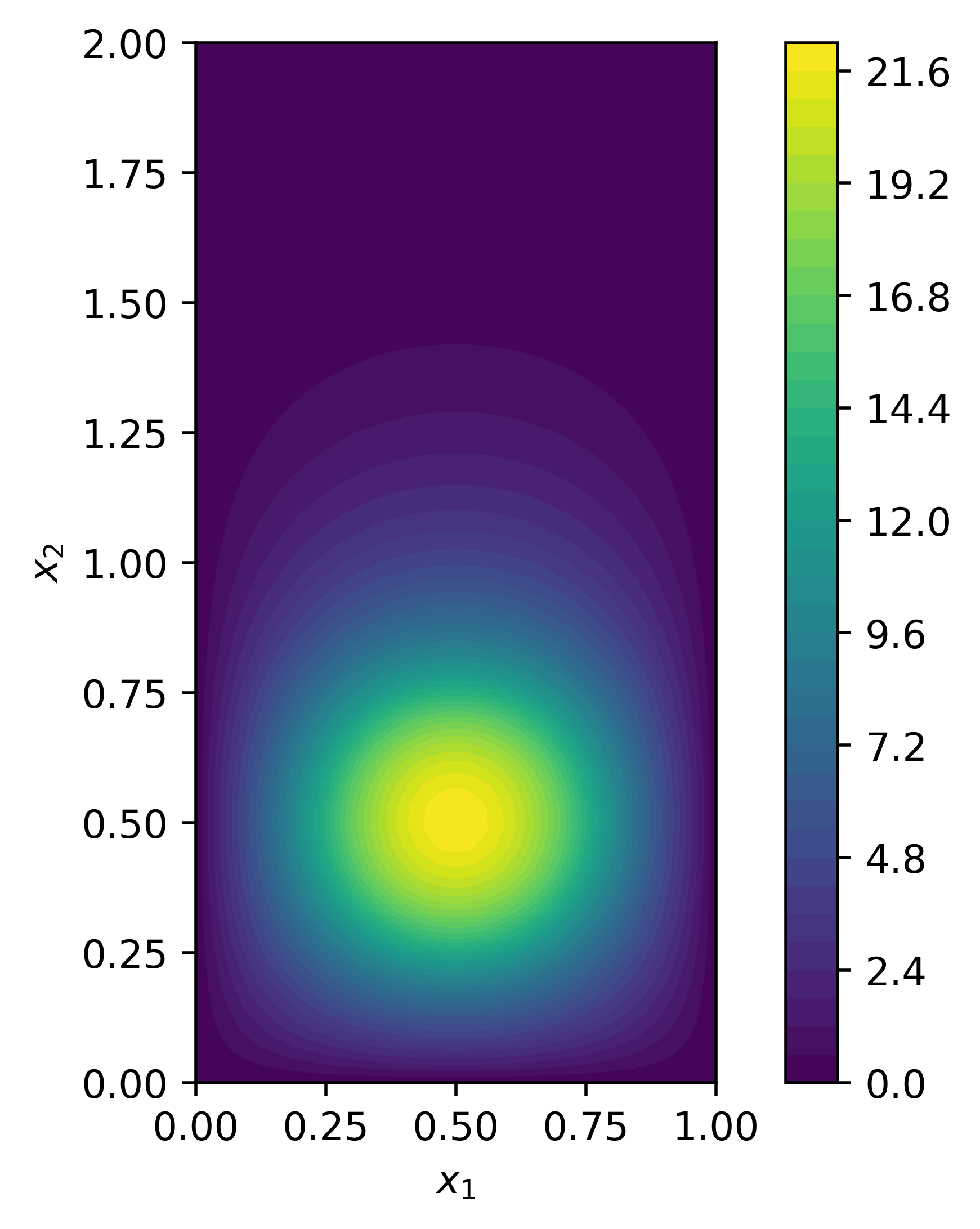}
        \hspace{0.001cm}
        \includegraphics[width=0.30\textwidth]{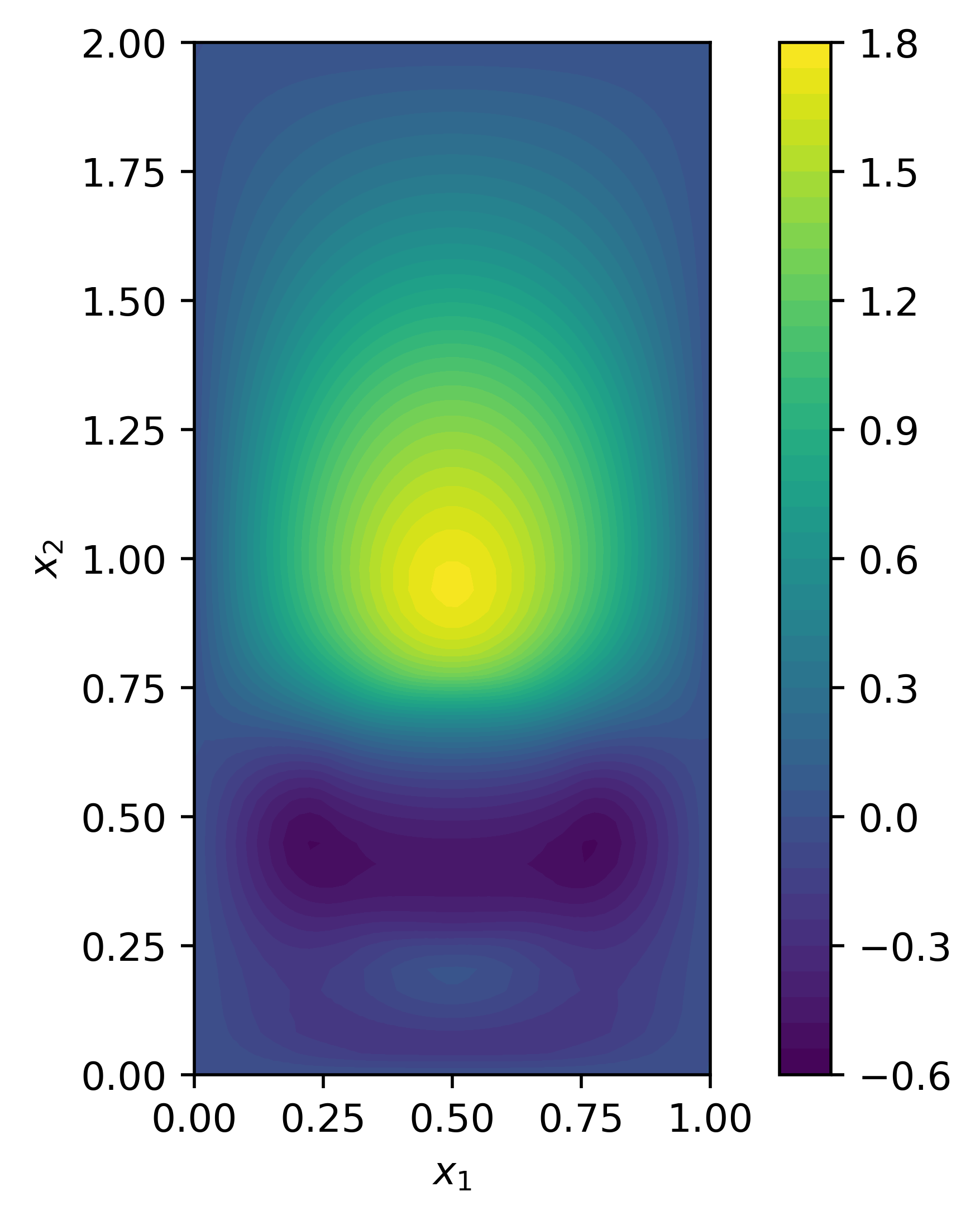}
		\caption{ Macro-solution with $G(u) = 1-2u$ corresponding to microscopic Geometry 1 (left),  Geometry 2 (middle), and the difference between these two solutions (right) at $T = 2$.}
		\label{macrosolution}
\end{figure}

\paragraph{Effect of the nonlinear drift}
We now investigate how the nonlinear drift interactions affect the macroscopic dispersion.
Here, all parameters are kept the same as in the previous numerical experiment, with the exception of the new nonlinearity $\Tilde{G}(u) = 1/(0.0001 + |1-2u|)$ which we compare with the previous $G(u) = 1-2u$.
The particular form of $\Tilde{G}$ is chosen to simulate the effect of high particle concentration impeding the flow. 
This can be seen by looking at \cref{fastdiffusion}, where low values of $|G(u)|$ or $\Tilde{G}(u)$ correspond to slightly slower dispersion than high values.

\begin{figure}[h!]
		\centering	
  \includegraphics[width=0.30\textwidth]{Pictures/Macro_T2_Geometry2.png}
  \hspace{0.001cm}
   \includegraphics[width=0.30\textwidth]{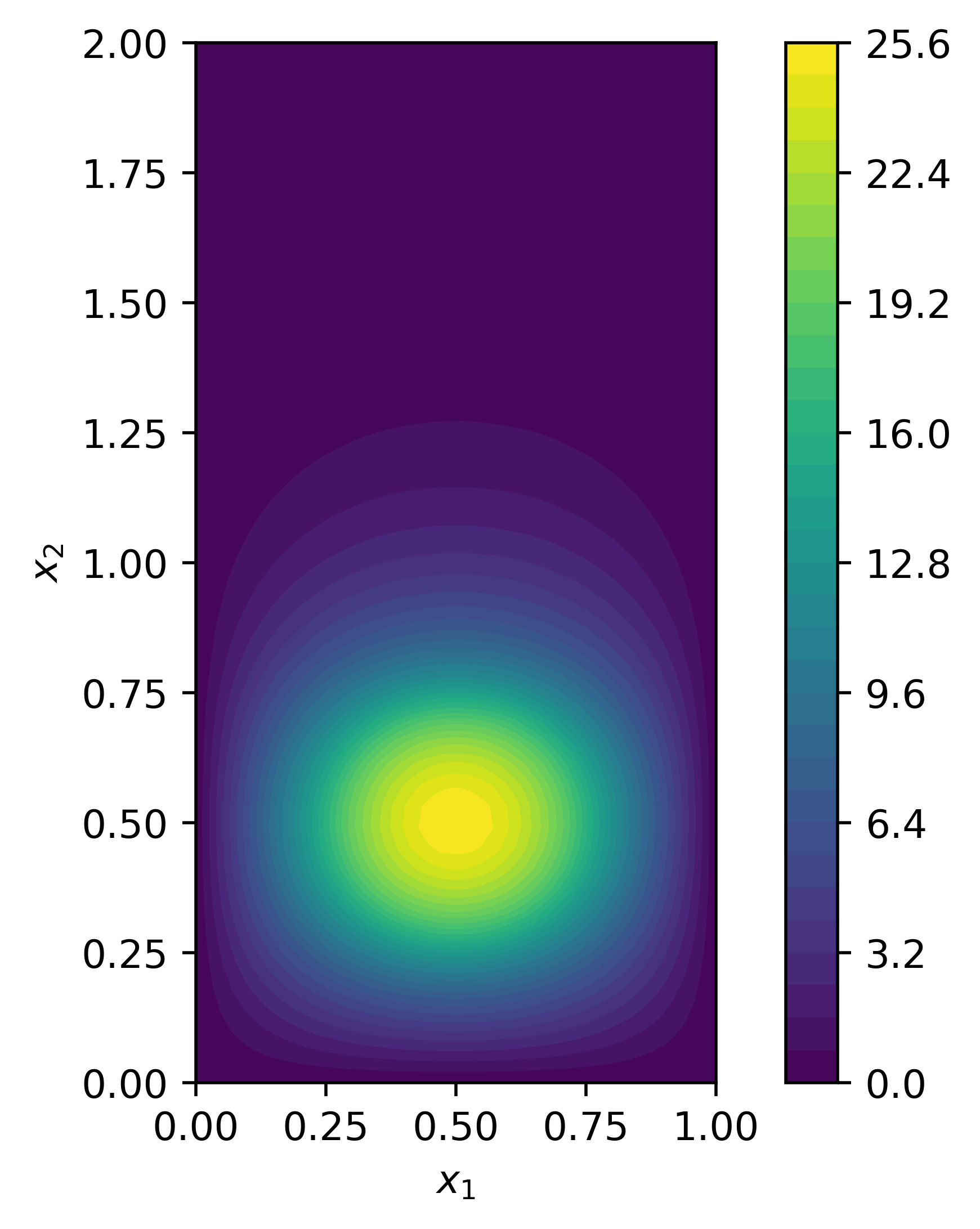}
  \hspace{0.001cm}
    \includegraphics[width=0.30\textwidth]{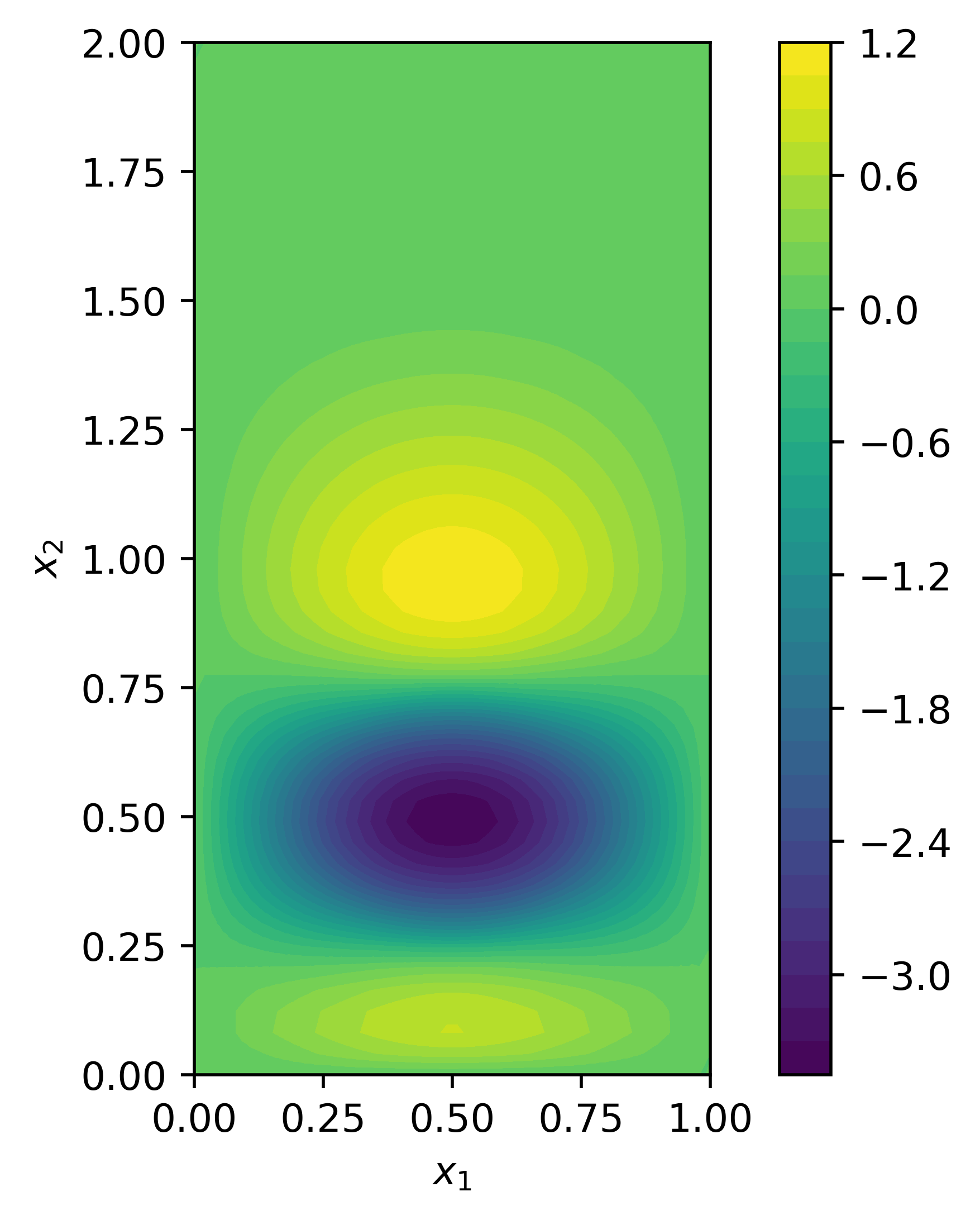}
		\caption{Macro-solution corresponding to Geometry 2 with $G(u) = 1-2u$ (left), $\Tilde{G}(u) = 1/(0.0001 + |1-2u|)$ (middle), and the difference between these two solutions (right) at $T=2$.}
		\label{macrosolution_diff_non}
\end{figure}

The macroscopic solution for both nonlinearities corresponding to Geometry 2 at $T=2$ is plotted in Figure \ref{macrosolution_diff_non}. 
Similar to before, the concentration diffuses slower with the nonlinearity, $\Tilde{G}(u)$, than with $G(u)$. 
Comparing the first and second plots in Figure \ref{macrosolution_diff_non}, we observe that the solution profile is more concentrated around the source in the case of $\Tilde{G}(u)$ (middle) than $G(u)$ (left). 
These effects can easily be seen in the right-most plot in Figure \ref{macrosolution_diff_non} where we calculate the difference between the concentration profiles. 
Similar observations can be made in the case of Geometry 1, however, with these choices of ingredients, the difference is less drastic, and so, we omit these plots.
As before, we examine the evolution of the mass in the upper subdomain, $M(t)$, and plot this in Figure \ref{mass_vertical_domain} on the right.

\begin{figure}[h!]
	\centering
    \begin{subfigure}[b]{0.48\textwidth}
        \centering
        \includegraphics[width=.85\textwidth]{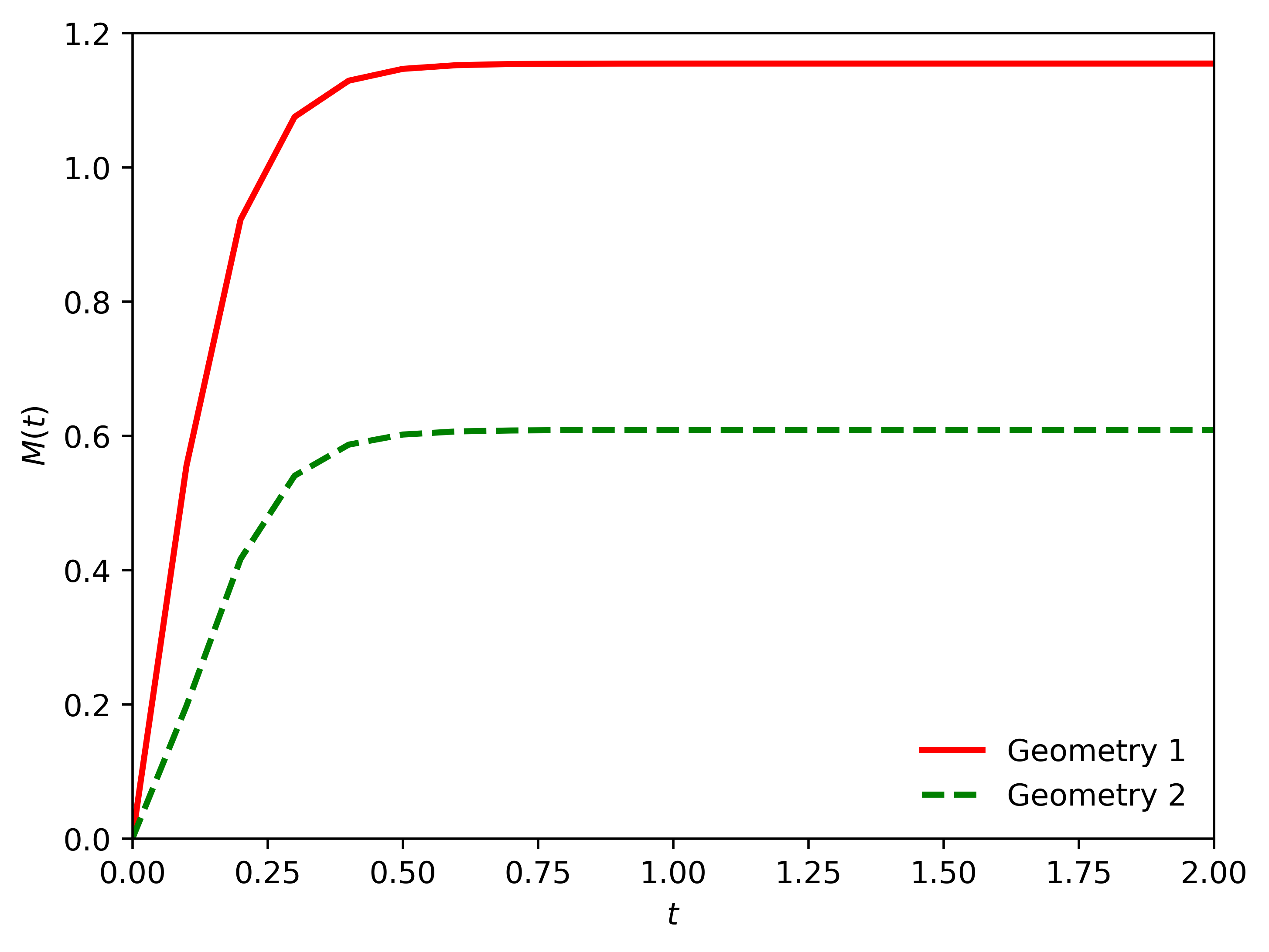}
    \end{subfigure}
    \hfill
    \begin{subfigure}[b]{0.48\textwidth}
        \centering
        \includegraphics[width=.85\textwidth]{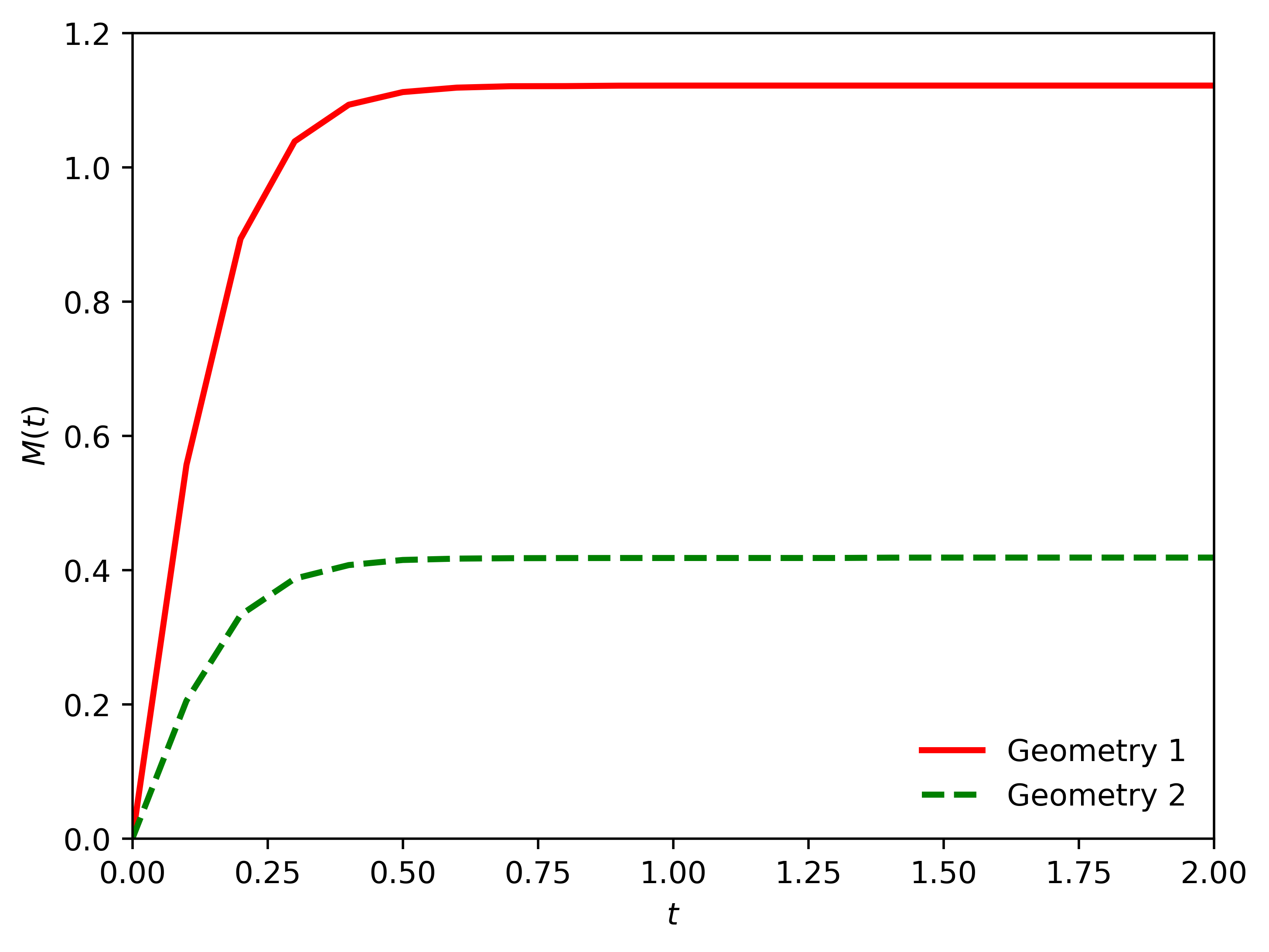}
    \end{subfigure}
		\caption{Comparison of mass $M(t)$ for Geometry 1 (solid red line) and Geometry 2 (dashed green line) with $G(u) = 1-2u$ (left) and $\Tilde{G}(u) = 1/(0.0001 + |1-2u|)$ (right).}
		\label{mass_vertical_domain}
\end{figure}

\newpage
\section{Conclusion and outlook}\label{section_conclusion}
Understanding the effect microscopic fast drifts have on the structure of the macroscopic dispersion is generally a difficult task, even in the case of simple microstructures underlying a regular porous medium. 
We have shown in this paper that solutions to our model (Problem $P(\Omega)$) not only exist but can also capture subtle two-scale effects of potential interest for applications. 
We will explore in follow-up works the numerical analysis of our setting as well as more numerical simulations of relevant practical case studies. 

By constructing a convergent iterative scheme to Problem $P(\Omega)$, we have shown the existence of weak solutions, as well as the well-posedness of linear iterative approximations to the original two-scale system. 
We have implemented the constructed iterative scheme using finite element methods and numerically illustrated some potential effects the microscopic geometry and the drift coupling have on the macroscopic dispersion tensor. 
Such {\em a posteriori} effects can be seen either by examining the different profiles of the macroscopic solutions or by following the time and space distribution of values in the entries of the effective dispersion tensor.  
We strongly believe that the working methodology proposed here can be used to approach a larger class of two-scale problems, covering among others the settings discussed, for instance, in
\cite{marusik2005,allaire2016,ijioma2019fast,raveendran2022upscaling}.

While the iteration scheme is a useful analytical tool, it is computationally expensive as the microscopic problem is solved for every time and space node in each iteration. 
However, given suitable assumptions about the choice of nonlinearity $G$, the dispersion tensor can be precomputed by solving the auxiliary problem \eqref{aux1}--\eqref{aux3} over a range of values for the parameter $p$ in a similar way to the computations in Section \ref{SSec:Influence_on_dispersion}.
Additionally, this precomputing can be parallelized, as the microscopic cell problems are independent and save considerable computation time. 
In light of this, we plan to study how this parallelization and precomputing technique affects the numerical results in terms of both the computation time and the order of accuracy. 
Of course, this requires us to additionally study the fully discrete error analysis of the Galerkin approximation (cf. e.g. \cite{lind2020semidiscrete, nepal2023analysis}) to determine the order of convergence in space and time.
Some early results in this direction can be found in \cite{nepal2024numerical}.

Within the frame of this manuscript, we considered that the macroscopic problem is defined in a two-dimensional space since our motivation originated from a hydrodynamic limit performed for an interacting particles system (TASEP) posed in two dimensions. 
So, it is natural to take $\Omega\subset \mathbb{R}^2$. 
However, without much additional difficulty, we can extend our problem and the corresponding results to dimension $3$ and higher. 
The only notable change is that, if Problem $P(\Omega)$ is posed in $\Omega\times Y\subset \mathbb{R}^n\times \mathbb{R}^d$, then there will be $d$ systems of nonlinear elliptic cell problems associated with the dispersion tensor. 
Additionally, the mathematical analysis work is not restricted to Dirichlet boundary conditions.
Other choices of boundary conditions (e.g. homogeneous Neumann) are allowed as well. 
Furthermore, the macroscopic equation does not need to be scalar. 
A system of macroscopic equations coupled correctly with a family of elliptic cell problems can be handled as well in a similar fashion.

\section*{Acknowledgments} We thank  H. Notsu (Kanazawa University, Japan) and M. Alam (Mahindra University, India) for fruitful discussions during their visit to Karlstad. 
The work of V.R., S.N., and A.M. is partially supported by the Swedish Research Council's project ``{\em  Homogenization and dimension reduction of thin heterogeneous layers}" (grant nr. VR 2018-03648).
The research activity of M.E. is funded by the European Union’s Horizon 2022 research and innovation program under the Marie Skłodowska-Curie fellowship project {\em{MATT}} (project nr.~101061956).
R.L. and A.M. are grateful to Carl Tryggers Stiftelse for their financial support via the grant CTS 21:1656.
 \bibliographystyle{amsplain}
	\bibliography{main}

\end{document}